\documentclass{article}
\usepackage[T1]{fontenc}
\usepackage[utf8]{inputenc}
\usepackage{times}
\usepackage[margin=1in]{geometry}
\usepackage{amsmath,amssymb,amsthm,mathtools,bm}
\usepackage{graphicx,booktabs,placeins,float}
\usepackage{microtype}
\usepackage[authoryear,round]{natbib}
\usepackage{hyperref}
\usepackage{url}
\hypersetup{hidelinks,pdftitle={Boundary Geometry and Global Rates of Cyclic Coordinate Descent for SVM Duals},pdfauthor={Egor Gladin and Eva Ibodova}}

\newcommand{\R}{\mathbb{R}}
\newcommand{\cB}{\mathcal{B}}
\DeclareMathOperator{\diag}{diag}
\DeclareMathOperator*{\argmin}{arg\,min}
\DeclareMathOperator*{\argmax}{arg\,max}
\newtheorem{theorem}{Theorem}
\newtheorem{lemma}[theorem]{Lemma}
\newtheorem{proposition}[theorem]{Proposition}

\title{Boundary Geometry and Global Rates of Cyclic Coordinate Descent for SVM Duals}
\author{Egor Gladin\\HSE University \and Eva Ibodova\\HSE University}
\date{}
\begin{document}
\maketitle

\begin{abstract}
Cyclic coordinate descent is widely used to train support vector machines, but global linear convergence alone gives limited guidance about the rate on a particular instance.
We study exact cyclic coordinate minimization for box-constrained convex quadratics and develop global contraction bounds that capture both coordinate order and boundary geometry.
Our convergence certificate combines two ingredients: interactions between successive coordinate updates, encoded by the triangular part of the normalized Hessian, and objective growth, captured by a matrix lower bound on the optimality gap.
It remains valid for cases with singular Hessians, nonunique minimizers, and updates clipped at the boundary.
An error-bound specialization of our analysis strictly sharpens the Wang--Lin factor when both use the same error-bound constant.
In the positive-definite case, the analysis shows that the worst-case contraction of one Gauss--Seidel sweep remains a global upper bound under box constraints, with equality for an interior optimizer.
We also demonstrate the limitations of Hessian-only bounds by constructing a two-coordinate family with a fixed singular normalized Hessian, whose worst-case contraction approaches one as its boundary margin vanishes.
The theory is complemented by practical procedures for bounding the worst-case contraction on individual instances and numerical evaluation on several classical SVM datasets.
\end{abstract}

\section{Introduction}

Box-constrained convex quadratics arise when a quadratic model is combined
with independent lower and upper bounds. A prominent example is the dual of a
support vector machine (SVM), for which cyclic coordinate minimization is a
standard training method \citep{hsieh2008}. Each update solves a scalar
quadratic and clips its minimizer to an interval. Without constraints, exact
cyclic coordinate minimization is the classical Gauss--Seidel iteration
\citep{varga2000,leewright2019}.

Worst-case rates over problem classes support algorithm comparisons
\citep{nemirovskiyudin1983,droriteboulle2014,taylorhendrickxglineur2017},
but need not predict behavior on a fixed quadratic instance.
For a fixed cyclic order, \citet{wanglin2014} prove global linear convergence
for the SVM dual even when its Hessian is singular, and give an explicit
contraction factor: a number $q<1$ such that the objective gap after a full
sweep is at most $q$ times the gap before the sweep. Their analysis uses a
residual error bound: for some $\kappa>0$, every feasible $x$ satisfies
$\operatorname{dist}(x,X_*)\leq\kappa\|R(x)\|$, where $X_*$ is the solution
set and $R(x)$ is a projected-gradient residual. The resulting contraction
factor is explicit in $\kappa$, but general scalar inequalities make it highly
conservative, and the analysis provides no practical way to evaluate
$\kappa$ on a given dataset. A useful instance-specific bound must account for
both the interactions accumulated during a sweep and how sharply the objective
rises as feasible points move away from the solution set.

The boundary is central to the second issue. At an optimal endpoint with a
nonzero partial derivative, moving a distance $t$ into the box incurs a
first-order penalty. Since $t$ is bounded by the box width, this penalty also
gives a quadratic lower bound, even in directions where a singular Hessian
has no curvature. As the endpoint margin vanishes, these directions can
become arbitrarily slow. Such growth is invisible to the positive eigenvalues
of the normalized Hessian alone. Globally, \emph{quadratic growth} means that
$f(x)-f_*\geq(\nu/2)\operatorname{dist}(x,X_*)^2$ for some $\nu>0$.

We develop an instance-specific analysis around this observation. A matrix
version of the growth bound retains its directional structure. The strictly
triangular part of the normalized Hessian describes the effect of updates made
later in a sweep. Combining the two matrices through a spectral quantity yields
what we call a \emph{spectral certificate} for global contraction; it retains
directional information that separate scalar bounds would lose.
Safe screening uses certified bounds on an unknown optimizer to prove that
selected coordinates must lie at box endpoints and can therefore be removed
from an optimization problem. We instead retain the quantitative slack in
these certificates. The resulting lower bounds on nonzero optimal derivatives
supply a computable growth matrix, including in some singular problems. The
spectral certificate concerns the original cyclic method, from its first
sweep. Screening supplies information for the analysis while the algorithm
keeps every coordinate and its original updates.

Our contributions have three parts.
\begin{itemize}
\item We derive the spectral certificate for exact cyclic minimization of
box-constrained convex quadratics
(Theorem~\ref{thm:matrix}). It allows singular Hessians, nonunique minimizers,
and clipped updates. Certified endpoint margins give a concrete instance of
the bound (Section~\ref{sec:boundary}). Its error-bound specialization
(Proposition~\ref{prop:error-bound}) strictly improves the factor obtained
from Theorem~8 and Lemma~10 of \citet{wanglin2014}, using the same error-bound
constant.
\item In the positive-definite case, we show that the exact one-sweep
objective contraction factor for unconstrained quadratic minimization---the
Gauss--Seidel factor---remains a global upper bound after box constraints are
imposed, and is \emph{attained} when the box-constrained problem has an
interior optimizer (Theorem~\ref{thm:pd}).
\item We construct a two-coordinate SVM family with a fixed singular
normalized Hessian and compute its exact worst-case contraction
(Theorem~\ref{thm:singular}). A vanishing boundary margin makes this factor
approach one and produces a trajectory requiring order $1/\delta$ sweeps
to halve its initial gap, where $\delta$ is the margin. Our spectral
certificate recovers the exact leading constant in the decrease fraction.
\end{itemize}

\subsection{Related work}

\paragraph{Convergence and cyclic order.}
Error bounds have long supported linear-convergence analyses beyond strongly
convex objectives \citep{luotseng1992,tseng2010}. \citet{wanglin2014}
establish global linear convergence for feasible descent methods and
explicitly cover exact cyclic CD on the SVM dual considered here. Their proof
gives an explicit factor in terms of descent and error-bound constants;
general scalar inequalities make the factor conservative, while evaluating
the error-bound constant on concrete data remains a separate task. Our
spectral certificate targets this quantitative gap by retaining matrix
structure and deriving computable geometric information from screening.
For strongly convex problems, \citet{li2018} sharpen cyclic block coordinate
complexity bounds using triangular truncation estimates.
\citet{sunye2021} show that cyclic CD can be much slower than randomized CD
on unconstrained quadratics. Their Claim~B.1 gives the contraction factor for
cyclic CD on a positive-definite quadratic without constraints, which arises
as a special case of our box-constrained analysis. Our contribution combines
this triangular matrix analysis with
boundary-induced growth under box constraints. Recent performance
estimation work by \citet{kamri2026} studies worst-case guarantees over
classes of unconstrained coordinate-wise smooth functions. Our bounds
instead use a specified quadratic, box, and coordinate order.

\paragraph{Growth, screening, and identification.}
For convex composite objectives, \citet{drusvyatskiylewis2018} show that
quadratic growth is equivalent, up to explicit constants, to an error bound in
which the proximal-gradient step length controls the distance to the solution
set; their statements hold on corresponding objective sublevel sets.
Safe SVM screening was developed by \citet{ogawa2013,ogawa2014}, and
\citet{zimmert2015} give dynamic screening based on a duality gap.
The Gap Safe framework uses gap-controlled regions to certify inactive
variables \citep{ndiaye2017}. We use such regions to retain \emph{lower
bounds on nonzero optimal derivatives}, and convert these margins into a
global growth matrix. This role differs from removing variables or
analyzing the iteration after identification. In particular,
\citet{klopfenstein2020} prove finite identification and local linear rates
for cyclic CD under regularity assumptions, while \citet{nutini2022}
quantify active-set identification complexity, including its dependence on
nondegeneracy margins. Our singular example isolates that dependence in an
exact global one-sweep factor and an explicit full trajectory.

\paragraph{Notation.}
We use $\|\cdot\|$ for the Euclidean norm and its induced matrix norm, and
$\lambda_{\min}$ and $\lambda_{\max}$ for the extreme eigenvalues of a
symmetric matrix. The notation $S\succeq0$ (respectively $S\succ0$) means
that $S$ is positive semidefinite (respectively positive definite), while
$S\succeq M$ means $S-M\succeq0$. For a nonempty closed convex set $P$,
$\operatorname{dist}(x,P):=\min_{z\in P}\|x-z\|$ is the distance to $P$,
and $\Pi_P$ denotes the Euclidean projection onto $P$. The matrix $I$
denotes the identity, $\diag$ forms a diagonal matrix, and
$(t)_+:=\max\{t,0\}$ denotes the positive part of $t$.

\section{Problem and cyclic updates}
\label{sec:problem}

We work with the diagonally normalized box quadratic
\begin{equation}
 \min_{x\in\cB} f(x),\qquad
 f(x)=\tfrac12x^\top Hx-a^\top x,\qquad
 \cB=\prod_{i=1}^p[0,r_i],
 \label{eq:problem}
\end{equation}
where $p\geq1$, $r_i>0$, $a\in\R^p$, and $H\in\R^{p\times p}$ is
symmetric positive semidefinite with $H_{ii}=1$.
For an SVM with hinge loss and no unregularized intercept, the original
dual takes the box-constrained form \citep{hsieh2008,wanglin2014}
$$
\min_{\alpha\in[0,C]^p}
\frac12\alpha^\top Q\alpha-\bm1^\top\alpha.
$$
Diagonal scaling gives \eqref{eq:problem}, preserves objective values and
exact coordinate updates, and yields $a_i r_i=C$; see
Appendix~\ref{app:scaling}. The analysis applies more generally to arbitrary
$a$ and positive side lengths. The coupled equality constraint introduced by
an unregularized intercept lies outside this setting.

Write $g(x)=Hx-a$ for the gradient, $X_*=\argmin_{x\in\cB}f(x)$ for the
solution set, $f_*=\min_{x\in\cB}f(x)$, and $\Delta(x)=f(x)-f_*$.
The set $X_*$ is nonempty, compact, and convex.
The unit-step projected-gradient residual is
\begin{equation}
 R(x):=x-\Pi_{\cB}(x-g(x)).
 \label{eq:residual}
\end{equation}

Fix the coordinate order $1,\ldots,p$. Starting at $x^0=x$, update
coordinate $i$ to obtain $x^i$, leaving all other coordinates unchanged:
\begin{equation}
 x_i^i=\Pi_{[0,r_i]}\bigl(x_i^{i-1}-g_i(x^{i-1})\bigr),
 \qquad i=1,\ldots,p.
 \label{eq:update}
\end{equation}
For an interval, $\Pi_{[0,r_i]}(t)=\min\{r_i,\max\{0,t\}\}$ is scalar
clipping.
Unit diagonal makes \eqref{eq:update} the exact scalar minimizer.
Denote one complete sweep by $\mathcal C(x)=x^p$. Our object of study is
the worst-case \emph{one-sweep objective contraction}
\begin{equation}
 \chi_*:=\sup_{x\in\cB\setminus X_*}
       \frac{\Delta(\mathcal C(x))}{\Delta(x)}.
 \label{eq:chi}
\end{equation}
A bound $\chi_*\leq q<1$ gives
$\Delta(\mathcal C^k(x))\leq q^k\Delta(x)$ for every starting point and
every integer $k\geq0$. A particular trajectory may have a smaller
asymptotic rate.

Split the Hessian according to the chosen order:
\begin{equation}
 H=I+T+T^\top,\qquad
 T_{ij}=\begin{cases}H_{ij},&i>j,\\0,&i\leq j.\end{cases}
 \label{eq:split}
\end{equation}
Changing the order permutes the problem data before this split. The matrix
$T^\top$ records how later updates change partial derivatives that were
already minimized.

\section{A spectral certificate}
\label{sec:matrix}

For $x\in\cB$, let $\bar x=\Pi_{X_*}(x)$ be its nearest optimizer in Euclidean distance.
Suppose a symmetric matrix $M\succeq0$ satisfies
\begin{equation}
 \Delta(x)\geq\tfrac12(x-\bar x)^\top M(x-\bar x)
 \qquad (x\in\cB).
 \label{eq:growth}
\end{equation}
We call any such $M$ a \emph{growth matrix}. This condition describes growth
away from the solution set and
permits nonunique minimizers even when $M\succ0$. For example, quadratic growth
with constant $\nu>0$ is \eqref{eq:growth} with $M=\nu I$.
Section~\ref{sec:boundary} constructs a matrix directly from boundary
margins.

\begin{theorem}
\label{thm:matrix}
Suppose \eqref{eq:growth} holds and $H+M\succ0$. For $0<\theta<1$, set
\begin{equation}
 A_\theta(M):=
 \frac{\lambda_{\max}\!\left(T(H+\theta M)^{-1}T^\top\right)}{1-\theta}.
 \label{eq:A}
\end{equation}
Then $H+\theta M\succ0$ and, for every $x\in\cB$,
\begin{align}
 \Delta(\mathcal C(x))
 &\leq\tfrac12 A_\theta(M)\|\mathcal C(x)-x\|^2,
 \label{eq:gap-step}\\
 \Delta(\mathcal C(x))
 &\leq\frac{A_\theta(M)}{1+A_\theta(M)}\Delta(x).
 \label{eq:rate}
\end{align}
\end{theorem}
\noindent\emph{Proof sketch.} The full proof appears in
Appendix~\ref{app:matrix-proof}. Write $y=\mathcal C(x)$, $s=y-x$, and
$e=y-\Pi_{X_*}(y)$. Coordinate optimality, including clipped updates,
gives $f(x)-f(y)\geq\|s\|^2/2$ and
$\Delta(y)\leq s^\top Te-e^\top He/2$. Apply \eqref{eq:growth}
at $y$ and complete the square in $e$; the descent inequality then
gives \eqref{eq:rate}.

Every admissible $\theta$ gives a valid bound, and minimizing
$A_\theta(M)$ over admissible $\theta$ yields the best spectral certificate
of this form. The matrix expression preserves the directions in which the
triangular interactions and the growth are large. Replacing those matrices
by separate norm and eigenvalue bounds can lose this information.

\subsection{Consequences and scope}

Every problem \eqref{eq:problem} has a constant $\nu>0$ such that
$\Delta(x)\geq(\nu/2)\operatorname{dist}(x,X_*)^2$.
Appendix~\ref{app:growth} proves this by describing the optimal set with
linear equations and applying Hoffman's bound \citep{hoffman1952}.
Thus Theorem~\ref{thm:matrix} always yields some global contraction, even
with a singular Hessian and nonunique solutions. The simpler estimate
$A_{1/2}(\nu I)\leq4\|T\|^2/\nu$ is available, but finding a useful
computable $M$ is the central quantitative task.

\begin{theorem}
\label{thm:pd}
Suppose $H\succ0$ and set
\begin{equation}
 A_0=\lambda_{\max}(TH^{-1}T^\top).
 \label{eq:pd}
\end{equation}
Then
\begin{equation}
 \chi_*\leq\frac{A_0}{1+A_0}.
 \label{eq:pd-rate}
\end{equation}
If the unique optimizer $x_*$ is interior to $\cB$, then \eqref{eq:pd-rate}
holds with equality.
\end{theorem}
The proof is given in Appendix~\ref{app:pd}. This factor is already implicit in
\citet[Claim~B.1]{sunye2021}: with $L=I+T$,
$L^\top H^{-1}L=I+TH^{-1}T^\top$, so their contraction factor for the
quadratic problem without box constraints,
$1-1/\lambda_{\max}(L^\top H^{-1}L)$ equals
$A_0/(1+A_0)$, the right-hand side of \eqref{eq:pd-rate}.
The result here shows that it remains an upper bound for a whole sweep
under box constraints, including sweeps with active clipping.

\begin{proposition}
\label{prop:error-bound}
Let $\kappa$ be any global residual error-bound constant. Then
\begin{equation}
 \chi_*\leq\frac{2\kappa\|T\|^2}{1+2\kappa\|T\|^2}.
 \label{eq:eb-rate}
\end{equation}
Moreover, if $\nu>0$ is any quadratic-growth constant, then
\begin{equation}
 \kappa_\nu:=1+\frac{2\|I-H\|}{\nu+\lambda_{\min}(H)}
 \label{eq:kappa-conversion}
\end{equation}
is a valid global residual error-bound constant. In particular,
\eqref{eq:eb-rate} holds with $\kappa=\kappa_\nu$.
\end{proposition}
Appendix~\ref{app:error-bound} proves the proposition and shows
that \eqref{eq:kappa-conversion} sharpens the general quadratic-growth-to-error-
bound conversion of \citet{drusvyatskiylewis2018} on the present problem class.
It also proves that \eqref{eq:eb-rate} strictly improves the Wang--Lin factor
for the same $\kappa$.

\section{Boundary margins give a computable growth matrix}
\label{sec:boundary}

This section first turns certified lower bounds on optimal-gradient
components into a growth matrix and then shows how to obtain such bounds
from any feasible reference point.

All optimizers have the same gradient, denoted by $g_*$
(Appendix~\ref{app:growth}). If $g_{*,i}>0$, every optimizer has
$x_{*,i}=0$; if $g_{*,i}<0$, every optimizer has $x_{*,i}=r_i$.
For now, suppose that certified margins are available:
\begin{equation}
 0\leq m_i\leq |g_{*,i}|,\qquad i=1,\ldots,p.
 \label{eq:margins}
\end{equation}
Coordinates with no certified margin receive $m_i=0$. These numbers
define the matrix
\begin{equation}
 M_{\rm scr}:=H+2\diag(m_1/r_1,\ldots,m_p/r_p),
 \label{eq:screen-matrix}
\end{equation}

To show that $M_{\rm scr}$ is a growth matrix,
fix any optimizer $x_*$ and write $d=x-x_*$.
The box optimality conditions give $g_{*,i}d_i\geq0$ for every $i$.
If $m_i>0$, the optimal coordinate is an endpoint, so $g_{*,i}d_i=|g_{*,i}|\,|d_i|$.
Since $|d_i|\leq r_i$,
\begin{equation}
 \Delta(x)=\tfrac12d^\top Hd+g_*^\top d
 \geq\tfrac12d^\top Hd+\sum_i\frac{m_i}{r_i}d_i^2
 =\tfrac12d^\top M_{\rm scr}d.
 \label{eq:screen-growth}
\end{equation}
Unlike \eqref{eq:growth}, which is stated for the nearest optimizer,
\eqref{eq:screen-growth} holds relative to \emph{every} optimizer.
It converts the linear penalty for leaving an optimal endpoint into a global
quadratic lower bound using only $H$ and the certified margins.

When $M_{\rm scr}\succ0$, substituting it in Theorem~\ref{thm:matrix}
gives a computable global rate. Let $F=\{i:m_i=0\}$ denote the remaining
coordinates. The precise applicability test is
\begin{equation}
 M_{\rm scr}\succ0\quad\Longleftrightarrow\quad H_{FF}\succ0,
 \label{eq:rank}
\end{equation}
where $H_{FF}$ is the principal submatrix on $F$ and the condition is
vacuous if $F$ is empty. Indeed, a vector has zero quadratic form under
$M_{\rm scr}$ exactly when it vanishes on all positive-margin coordinates
and has zero quadratic form under $H$.
This connects the test to the restricted-curvature conditions used in
local CD analyses \citep{klopfenstein2020}; here it certifies a global
bound before the iterates identify those endpoints.

There is a real limitation: $M_{\rm scr}\succ0$ implies uniqueness, by
putting a second optimizer into \eqref{eq:screen-growth}.
Exact dual nonuniqueness is nevertheless expected to be uncommon in practice:
for generic SVM data, the box constraints prevent Hessian null directions
from remaining feasible along the optimal face. It arises mainly from
degeneracies such as duplicated or linearly dependent margin examples.
On a nonunique instance, test \eqref{eq:rank} necessarily fails, so this
screening construction is inconclusive; Theorem~\ref{thm:matrix} still
applies with an appropriate growth matrix.

\subsection{Margins from one feasible point}
\label{sec:fw}

We now construct the margins assumed above from any feasible reference
point $z$. Its Frank--Wolfe gap \citep{jaggi2013} is
\begin{equation}
 G_{\rm FW}(z):=\max_{v\in\cB}g(z)^\top(z-v)
 =\sum_i\bigl[z_i(g_i(z))_++(r_i-z_i)(-g_i(z))_+\bigr].
 \label{eq:fw}
\end{equation}
For any optimizer, set $d=z-x_*$. Then
\[
 G_{\rm FW}(z)\geq g(z)^\top d
 =d^\top Hd+g_*^\top d\geq d^\top Hd.
\]
Cauchy--Schwarz applied after multiplication by the positive semidefinite
square root of $H$ gives
$|g_i(z)-g_{*,i}|^2\leq H_{ii}d^\top Hd\leq G_{\rm FW}(z)$.
Consequently
\begin{equation}
 m_i=\bigl(|g_i(z)|-\sqrt{G_{\rm FW}(z)}\bigr)_+
 \label{eq:fw-margin}
\end{equation}
satisfies \eqref{eq:margins}. A positive margin also certifies the sign
of $g_{*,i}$, and hence its endpoint. This is a specialization of
gap-based safe screening \citep{zimmert2015,ndiaye2017}; the ensuing
matrix \eqref{eq:screen-matrix} uses the \emph{size} of the certified
margin to bound the rate.

The point $z$ may come from cyclic CD or another optimizer. Once computed,
the resulting global bound applies to every starting point of the original
cyclic method. Better information about the optimum improves the
certificate while the cyclic method itself remains fixed.

\subsection{Refinements and limits}

Appendix~\ref{app:screening} develops complementary margins from exact
gradient ranges on certified faces and from objective lower bounds. It
also gives an auxiliary concave maximization for improving those lower
bounds and a stronger matrix on a known objective sublevel set.
Retaining the largest previously certified margin for each coordinate
makes $M_{\rm scr}$ nondecreasing in the positive semidefinite order.
For fixed $\theta$, the coefficient $A_\theta(M_{\rm scr})$ then cannot
increase, because inversion reverses that order for positive definite
matrices.

The spectral bound is at least as sharp as replacing $M_{\rm scr}$ by its
smallest eigenvalue times $I$. Converting that same eigenvalue to the
error-bound constant in \eqref{eq:kappa-conversion} also cannot
improve the optimized spectral bound; Appendix~\ref{app:dominance} proves
the comparison. A separately obtained, sharper residual error bound can
still be useful.

\section{A fixed Hessian with arbitrarily slow contraction}
\label{sec:singular}

The next family shows why boundary information is needed and tests the
sharpness of the certificate. For $0<\delta\leq1/4$, let
\begin{equation}
 H=\begin{bmatrix}1&1\\1&1\end{bmatrix},\qquad
 a=\begin{bmatrix}1\\1+\delta\end{bmatrix},\qquad
 r=\frac1{1+\delta},\qquad
 \cB=[0,1]\times[0,r].
 \label{eq:family}
\end{equation}
These are normalized SVM data with $C=1$ and original Hessian
$Q=\bigl[\begin{smallmatrix}1&r\\r&r^2\end{smallmatrix}\bigr]$.
Throughout the family, $H$ has rank one, positive
eigenvalue $2$, and $\|T\|=1$. Both box side lengths stay at least $4/5$.

\begin{theorem}
\label{thm:singular}
For \eqref{eq:family}, the unique optimizer is $x_*=(1-r,r)^\top$, with
$g_*=(0,-\delta)^\top$. For coordinate order $1,2$, the exact worst-case
contraction factor is
\begin{equation}
 \chi_*=1-\frac{\delta}{2r}.
 \label{eq:family-exact}
\end{equation}
The point $(1,0)^\top$ attains this worst-case ratio. Starting there,
cyclic CD first reaches the optimizer after exactly
$\lceil r/\delta\rceil+1$ sweeps, and more than
$\lfloor r/(2\delta)\rfloor$ sweeps are needed to halve its initial gap.
Using the exact margin $m_2=\delta$ in \eqref{eq:screen-matrix}, the
single choice $\theta=1/2$ in Theorem~\ref{thm:matrix} gives
\begin{equation}
 \chi_*\leq 1-\frac{3\delta}{6r+7\delta},\qquad
 \frac{3\delta/(6r+7\delta)}{1-\chi_*}\longrightarrow1
 \quad(\delta\downarrow0).
 \label{eq:family-bound}
\end{equation}
\end{theorem}

The proof in Appendix~\ref{app:singular} computes the gap and the entire
cycle explicitly:
\begin{equation}
 \Delta(x)=\tfrac12(x_1+x_2-1)^2+\delta(r-x_2),\qquad
 \mathcal C(x)=\bigl(1-x_2,\min\{x_2+\delta,r\}\bigr)^\top.
 \label{eq:family-map}
\end{equation}
Along the Hessian null direction, the gap comes entirely from the
boundary penalty. At $(1,0)^\top$ the gap is $\delta r$, whereas the
first sweep decreases it by only $\delta^2/2$. Later sweeps move the
second coordinate by $\delta$ until it reaches its endpoint, explaining
the long trajectory.

Thus no factor strictly below one depending only on the normalized Hessian
can hold uniformly over these SVM instances. This complements the
unconstrained quadratic examples of \citet{sunye2021} and the
margin-dependent identification bounds of \citet{nutini2022}: dimension
and all normalized Hessian eigenvalues are fixed here, and both the
global contraction and a slow trajectory are known exactly.
The last limit in \eqref{eq:family-bound} establishes sharpness of the
\emph{decrease fraction}, a stronger comparison than merely observing
that both factors tend to one.

At $\delta=0$, the optimal set expands to the segment $x_1+x_2=1$ and
one sweep solves the problem. This expansion explains the discontinuity:
for positive $\delta$ the initial gap itself tends to zero, and the ratio in
\eqref{eq:family-exact} is not uniform through the change in the solution
set. Thus some singular instances with zero margins still converge in one
sweep.

\section{Experiments}
\label{sec:experiments}

We evaluate the size of the one-sweep bounds on six scaled binary SVM
datasets from the LIBSVM collection \citep{changlin2011libsvm}.
The code for the experiments is publicly available.\footnote{\url{https://github.com/evaibodova/cyclic-cd-rate}}
This section reports sonar, heart, and australian; Appendix~\ref{app:experiments}
contains diabetes, fourclass, and german.numer, together with the full
protocol. For each dataset we use the linear kernel and the RBF kernel
$K(x,x')=\exp(-\gamma\|x-x'\|^2)$, with
$C\in\{10^{-2},10^{-1},1,10,10^2\}$ and, for RBF,
$\gamma/\gamma_0\in\{1/10,1/3,1,3,10\}$. Here $\gamma_0$ is the
reciprocal of the median nonzero squared pairwise distance. We analyze the
original data order and the dual without an unregularized intercept.

The spectral certificate gives an upper bound on $\chi_*$. For feasible
$x,z$ with $f(\mathcal C(x))>f(z)$, a complementary witness bound is
\begin{equation}
 \frac{f(\mathcal C(x))-f(z)}{f(x)-f(z)}
 \leq\chi_*\leq\frac{A_\theta(M_{\rm scr})}{1+A_\theta(M_{\rm scr})},
 \label{eq:bracket}
\end{equation}
when the screened matrix is admissible. Appendix~\ref{app:witness}
gives the lower-bound argument, witness search, and verification details.

Let $\chi_{\rm scr}$ denote the screening-based spectral upper bound on
$\chi_*$ obtained by minimizing \eqref{eq:A} over $\theta$, and let
$\chi_{\rm obs}$ denote the largest witness ratio found by a
multistart local search. For the same screened growth constant we also
evaluate the Wang--Lin upper bound $\chi_{\rm WL}$ from \eqref{eq:WL};
when $H$ passes a positive-definiteness check we evaluate the
Gauss--Seidel upper bound $\chi_{\rm GS}$ from \eqref{eq:pd-rate}.
We plot the
\emph{decrease fractions} $1-\chi$: a larger lower bound means a stronger
contraction guarantee.

\paragraph{Linear SVM.}
Table~\ref{tab:linear-main} gives the linear-kernel decrease fractions.
The linear Hessians are singular, so the positive-definite
Gauss--Seidel formula does not apply. Across all six datasets, screening
gives a positive guarantee in 24 of the 30 linear configurations.
For sonar and heart, screening
provides a positive guarantee throughout the $C$ grid and is several
orders of magnitude stronger than Wang--Lin, although still below the
observed search values. For australian, the screened growth matrix
does not pass the numerical positive-definiteness test at any $C$;
the Wang--Lin constant computed from that matrix is therefore also
unavailable.

\begin{table}[H]
\centering
\small
\begin{tabular}{llrrr}
\toprule
Dataset & $C$ & $1-\chi_{\rm obs}$ & $1-\chi_{\rm scr}$ & $1-\chi_{\rm WL}$ \\
\midrule
Sonar & $0.01$ & $8.23\times10^{-2}$ & $1.03\times10^{-3}$ & $2.49\times10^{-11}$ \\
 & $0.1$ & $1.39\times10^{-2}$ & $9.90\times10^{-5}$ & $4.68\times10^{-12}$ \\
 & $1$ & $1.32\times10^{-3}$ & $8.84\times10^{-6}$ & $3.84\times10^{-13}$ \\
 & $10$ & $4.38\times10^{-4}$ & $1.05\times10^{-6}$ & $1.69\times10^{-14}$ \\
 & $100$ & $4.50\times10^{-5}$ & $7.74\times10^{-11}$ & $7.91\times10^{-19}$ \\
\midrule
Heart & $0.01$ & $1.43\times10^{-1}$ & $4.52\times10^{-4}$ & $4.00\times10^{-11}$ \\
 & $0.1$ & $5.32\times10^{-2}$ & $3.10\times10^{-5}$ & $3.42\times10^{-12}$ \\
 & $1$ & $6.38\times10^{-3}$ & $5.77\times10^{-7}$ & $7.34\times10^{-14}$ \\
 & $10$ & $6.95\times10^{-3}$ & $4.42\times10^{-9}$ & $8.28\times10^{-17}$ \\
 & $100$ & $1.14\times10^{-2}$ & $4.55\times10^{-9}$ & $4.53\times10^{-16}$ \\
\midrule
Australian & $0.01$ & $8.79\times10^{-2}$ & -- & -- \\
 & $0.1$ & $4.56\times10^{-2}$ & -- & -- \\
 & $1$ & $4.65\times10^{-2}$ & -- & -- \\
 & $10$ & $9.84\times10^{-3}$ & -- & -- \\
 & $100$ & $2.51\times10^{-2}$ & -- & -- \\
\bottomrule
\end{tabular}
\caption{Linear-SVM decrease fractions for the three main-text datasets.
The observed column is a search-based witness value; scr and WL are
floating-point evaluations of global upper bounds on $\chi_*$. For
Australian, screening fails its numerical growth test at every $C$,
so both corresponding columns are unavailable.}
\label{tab:linear-main}
\end{table}

\paragraph{RBF bounds.}
Figure~\ref{fig:rbf-main} shows the observed and screened decrease fractions
across the complete $5\times5$ grid for each main-text dataset. The
screened decrease fraction is positive in 73 of the 75 configurations;
the two missing cases are australian at $\gamma=\gamma_0/10$ and
$C\in\{10,100\}$, where the numerical positive-definiteness test for
the screened growth matrix fails. The distance between the plotted values
shows that the certificate can remain conservative even when it is much
more informative than a scalar bound.

\begin{figure}[t]
\centering
\includegraphics[width=.86\linewidth]{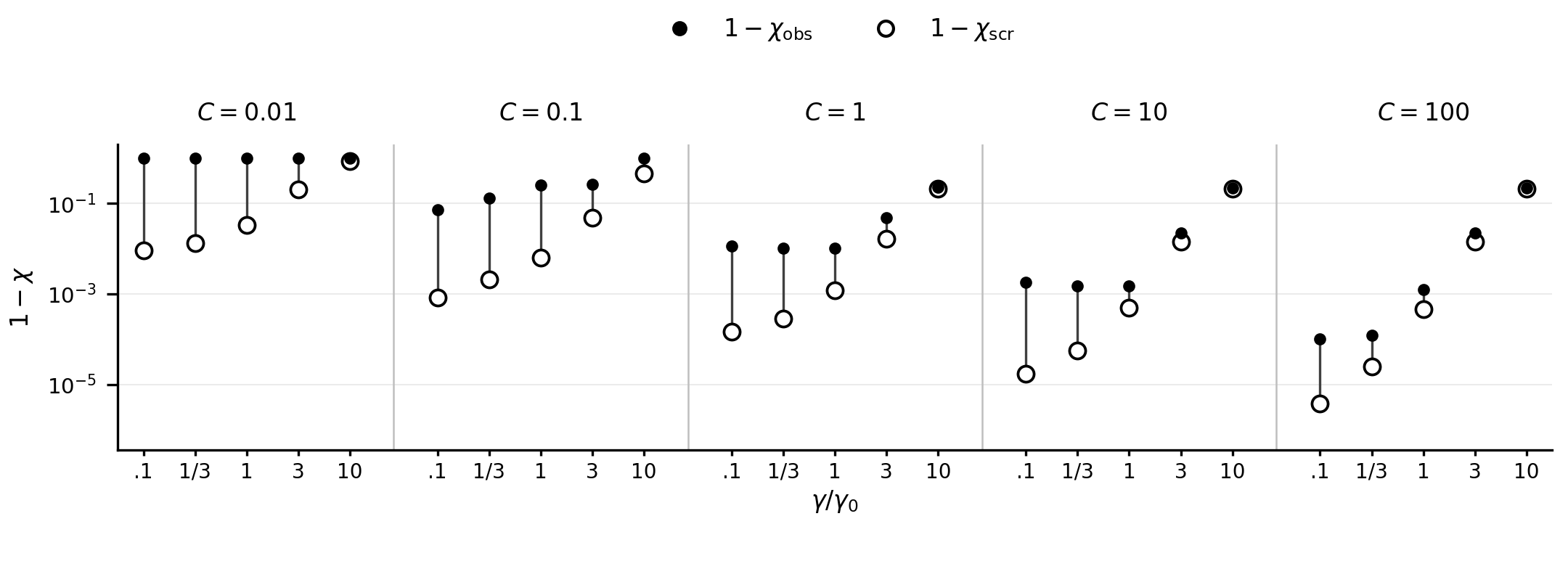}
\includegraphics[width=.86\linewidth]{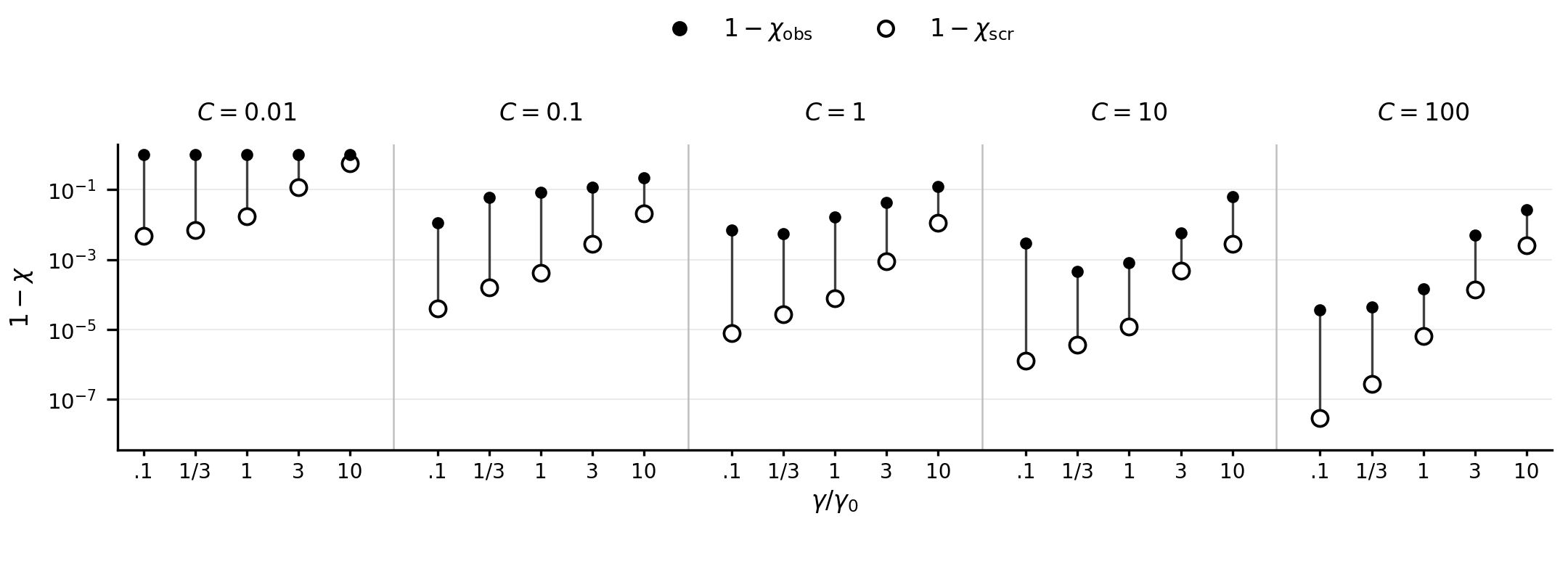}
\includegraphics[width=.86\linewidth]{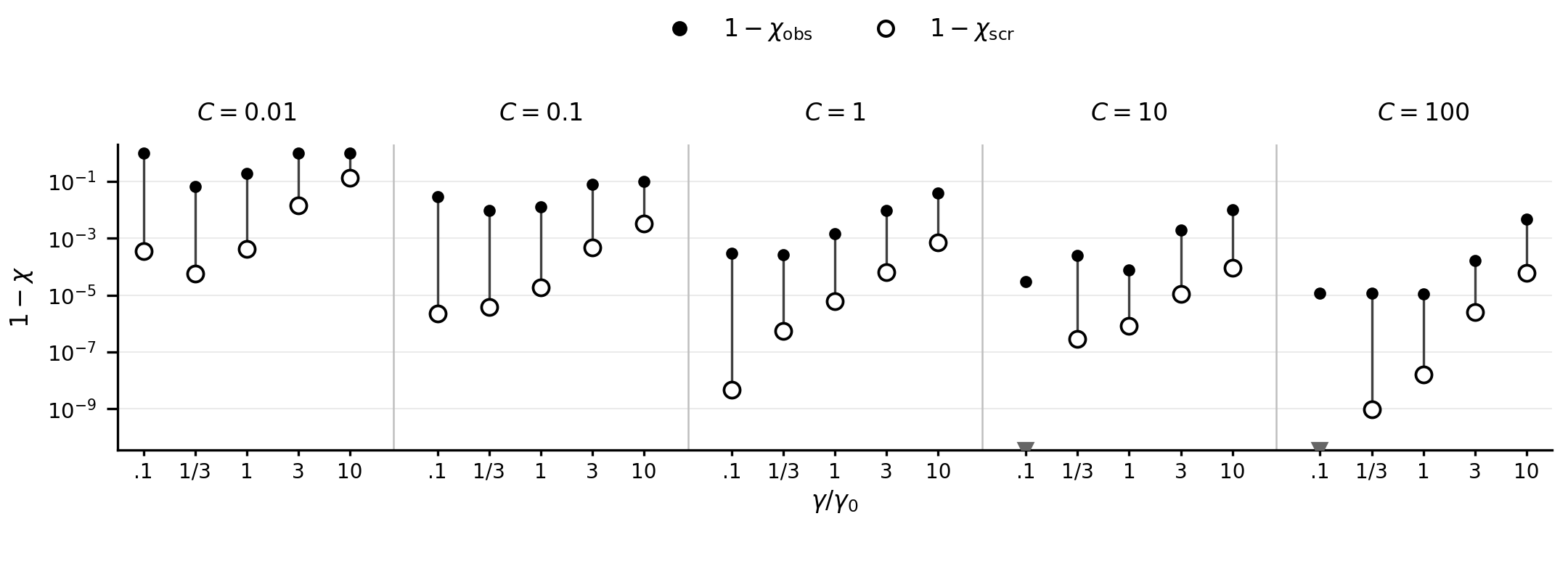}
\caption{RBF one-sweep decrease fractions for sonar (top), heart (middle),
and australian (bottom). Filled circles show $1-\chi_{\rm obs}$ and open
circles show $1-\chi_{\rm scr}$; vertical segments join the two values.
Horizontal ticks give $\gamma/\gamma_0$ within each labeled $C$ group.
The vertical scale is logarithmic. Missing open circles
indicate unavailable screened certificates.}
\label{fig:rbf-main}
\end{figure}

Figure~\ref{fig:heatmaps-main} compares guaranteed decrease fractions
with Wang--Lin by plotting
$\log_{10}[(1-\chi_{\rm scr})/(1-\chi_{\rm WL})]$.
The strict same-constant comparison follows from
Proposition~\ref{prop:error-bound} and Appendix~\ref{app:dominance}.
Across the 73 available pairs, the median ratio of decrease fractions
is $1.7\times10^7$. The corresponding median ratio against
Gauss--Seidel over 60 available pairs is $17.3$.

\FloatBarrier
\begin{figure}[H]
\centering
\includegraphics[width=.9\linewidth]{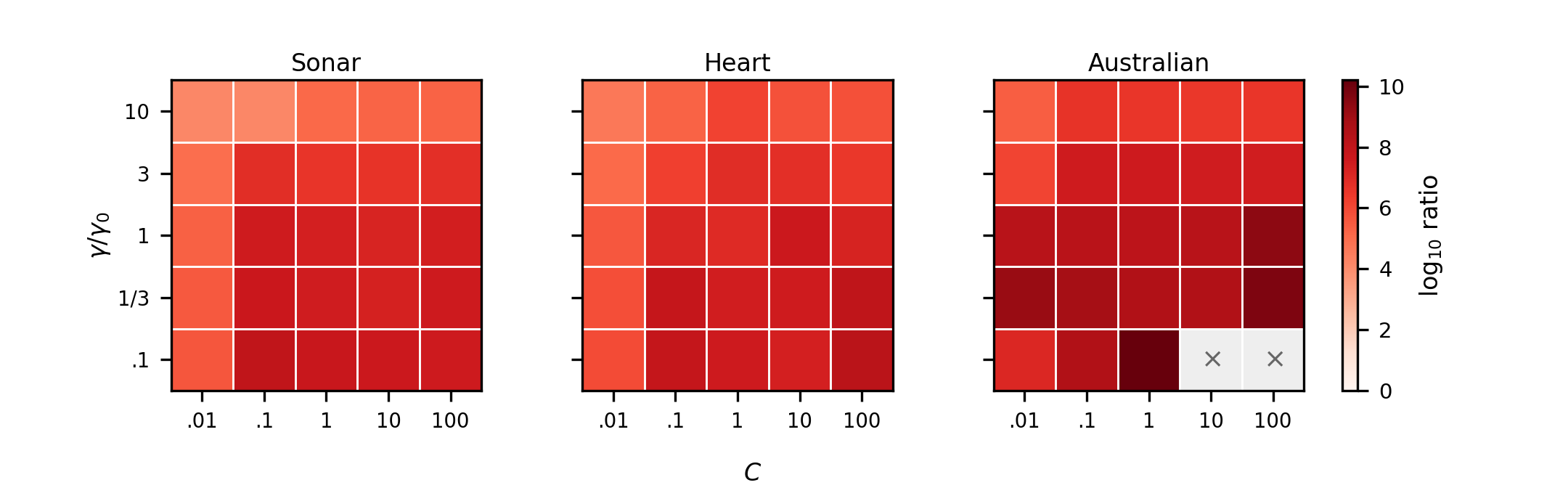}
\caption{RBF certificate improvement on Wang--Lin for sonar, heart, and
australian. Each cell is the base-10 logarithm of the ratio of guaranteed
decrease fractions. Red cells favor screening; a cross marks an
unavailable comparison.}
\label{fig:heatmaps-main}
\end{figure}

\paragraph{Coordinate order.}
Figure~\ref{fig:permutations}
compares 100 seeded coordinate permutations on sonar in two RBF settings
and one singular linear setting. Both observed and screened decrease
fractions vary with order, while the Wang--Lin value is essentially
order invariant. This shows order sensitivity; selecting an order by
its certificate would require separate evaluation.

\begin{figure}[H]
\centering
\includegraphics[width=.85\linewidth]{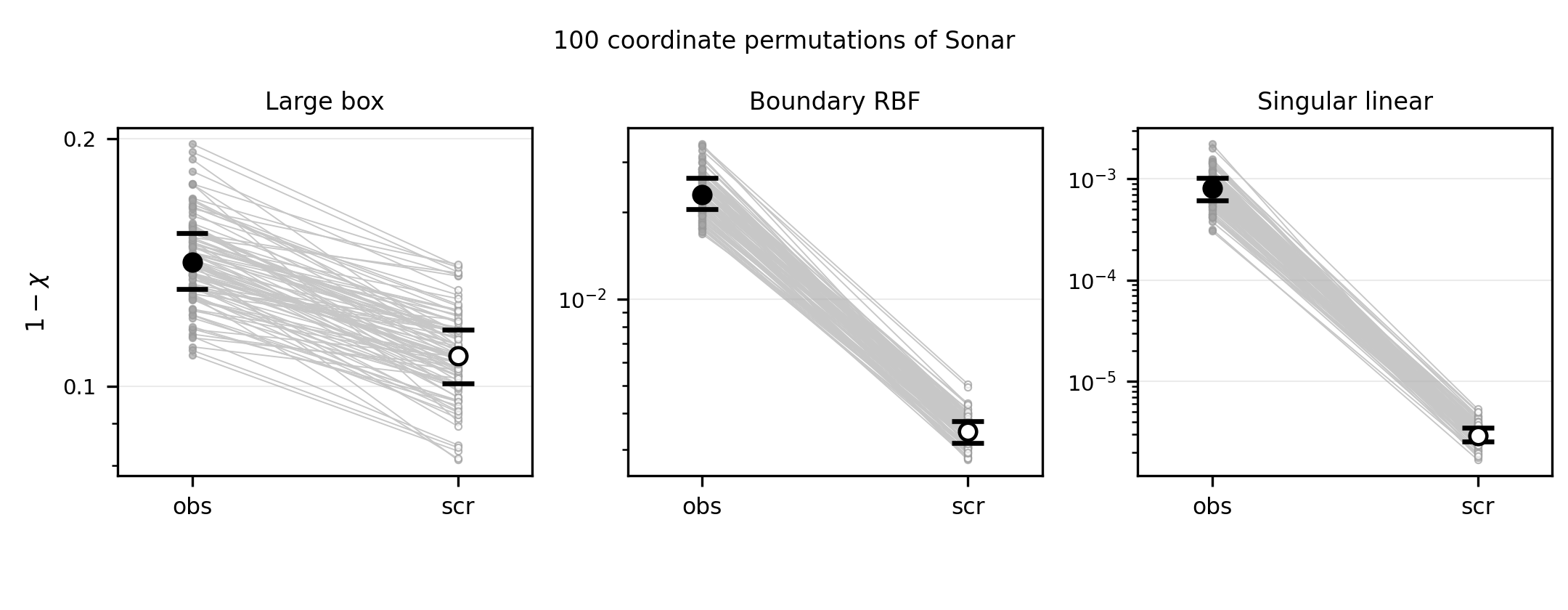}
\caption{Effect of 100 coordinate permutations on sonar. Left: large-box
RBF, $C=100$, $\gamma=10\gamma_0$, with a mostly free optimum;
middle: RBF, $C=1$, $\gamma=3\gamma_0$, with many active bounds; right:
linear SVM, $C=1$, with a singular Hessian. Gray lines pair the two
decrease-fraction estimates for the same permutation. Large markers
show medians and black bars show interquartile ranges; each panel has
its own logarithmic vertical scale.}
\label{fig:permutations}
\end{figure}

\section{Discussion}

Boundary margins can supply objective growth in directions missed by a
singular Hessian. Keeping that growth as a matrix and combining it with
the cyclic order gives the spectral certificate, whose decrease
fraction is asymptotically exact on a singular SVM family. The result
separates existence of a linear rate from the harder task of certifying
its magnitude on an instance. The main remaining limitations are the
possible failure of the screened matrix test, the conservatism of other
instances' bounds, and the cost of computing them. Extending the analysis
to coupled dual constraints and establishing whether the bounds reliably
guide coordinate ordering are natural next questions.

\subsection*{AI use statement}
Generative AI was used to improve wording, correct language errors, assist
with additional literature searches, and support proof development under
careful author guidance. The authors remain fully responsible for the
manuscript.

\subsection*{Reproducibility statement}
The appendix contains the proofs and a full description of the experiments.
The code and scripts for reproducing the experiments are linked in
Section~\ref{sec:experiments}.

\bibliographystyle{plainnat}
\bibliography{references}
\clearpage
\appendix
\numberwithin{equation}{section}
\numberwithin{theorem}{section}
\section{Proof of the spectral certificate}
\label{app:matrix-proof}

Fix one sweep, and write $y=\mathcal C(x)$ and $s=y-x$. Coordinate $i$
changes only once, by $s_i$. Its optimality condition is
\begin{equation}
 g_i(x^i)(z_i-y_i)\geq0\qquad(z_i\in[0,r_i]).
 \label{eq:coordinate-optimality}
\end{equation}
Exact quadratic expansion, followed by
$g_i(x^i)=g_i(x^{i-1})+s_i$, gives
\[
 f(x^{i-1})-f(x^i)
 =-g_i(x^{i-1})s_i-\tfrac12s_i^2
 =-g_i(x^i)s_i+\tfrac12s_i^2\geq\tfrac12s_i^2.
\]
The last inequality uses \eqref{eq:coordinate-optimality} with the old
feasible value $z_i=x_i^{i-1}$. Summing over the sweep yields
\begin{equation}
 f(x)-f(y)\geq\tfrac12\|s\|^2.
 \label{eq:decrease}
\end{equation}

Let $\bar y=\Pi_{X_*}(y)$ and $e=y-\bar y$. After update $i$, only
coordinates $j>i$ can change partial derivative $i$. Therefore
\begin{equation}
 g_i(y)-g_i(x^i)=\sum_{j>i}H_{ij}s_j=(T^\top s)_i.
 \label{eq:chronology}
\end{equation}
Using $z_i=\bar y_i$ in \eqref{eq:coordinate-optimality} gives
$g_i(x^i)e_i\leq0$. Exact expansion around $y$ and
\eqref{eq:chronology} now imply
\begin{align}
 \Delta(y)
 &=g(y)^\top e-\tfrac12e^\top He \notag\\
 &=s^\top Te+\sum_i g_i(x^i)e_i-\tfrac12e^\top He
 \leq s^\top Te-\tfrac12e^\top He.
 \label{eq:end-gap}
\end{align}
This step accounts for clipped updates through their optimality signs.

Set $G=H+\theta M$. Since
$G=(1-\theta)H+\theta(H+M)\succ0$, its inverse exists.
Subtracting $\theta$ times the lower bound
$\Delta(y)\geq e^\top Me/2$ from \eqref{eq:end-gap} yields
\[
 (1-\theta)\Delta(y)\leq s^\top Te-\tfrac12 e^\top Ge.
\]
Complete the square:
\begin{align*}
 s^\top Te-\tfrac12e^\top Ge
 &=\tfrac12s^\top TG^{-1}T^\top s
   -\tfrac12(e-G^{-1}T^\top s)^\top G(e-G^{-1}T^\top s)\\
 &\leq\tfrac12\lambda_{\max}(TG^{-1}T^\top)\|s\|^2.
\end{align*}
Division by $1-\theta$ proves \eqref{eq:gap-step}. Combining it with
\eqref{eq:decrease} gives
\[
 \Delta(y)\leq A_\theta(M)\bigl(\Delta(x)-\Delta(y)\bigr),
\]
which rearranges to \eqref{eq:rate}, including when $s=0$.
\hfill$\square$

\section{Normalization and solution geometry}
\label{app:scaling}

\subsection{From the SVM dual to the normalized quadratic}

For labeled examples $(v_i,y_i)$, $i=1,\ldots,p$, with
$y_i\in\{-1,1\}$ and a positive semidefinite kernel $K$, put
$Q_{ij}=y_i y_j K(v_i,v_j)$. The hinge-loss SVM without an
unregularized intercept has dual minimization form
\begin{equation}
 \min_{\alpha\in[0,C]^p}q(\alpha),\qquad
 q(\alpha)=\tfrac12\alpha^\top Q\alpha-\bm1^\top\alpha,
 \label{eq:svm}
\end{equation}
where $C>0$ and $\bm1$ is the all-ones vector. This is the box dual
used in coordinate-descent SVM training \citep{hsieh2008,wanglin2014}.
A bias incorporated as a regularized feature can be included in the
kernel. A separate unregularized bias instead imposes an equality
constraint on $\alpha$, outside our setting.

Assume first that $Q_{ii}>0$ for every coordinate. Set
\begin{equation}
 \begin{gathered}
 D=\diag(Q_{11},\ldots,Q_{pp}),\qquad x=D^{1/2}\alpha,\\
 H=D^{-1/2}QD^{-1/2},\qquad a=D^{-1/2}\bm1,\qquad
 r_i=C\sqrt{Q_{ii}}.
 \end{gathered}
 \label{eq:normalization}
\end{equation}
Then $q(\alpha)=f(x)$, the box becomes $\prod_i[0,r_i]$, and
$H_{ii}=1$. Because the change of variables rescales each coordinate
independently by a positive number, exact minimization along a coordinate
commutes with this transformation. Every cycle and its objective ratio
are preserved. In particular, $a_i r_i=C$.

If $Q_{ii}=0$, positive semidefiniteness gives
$|Q_{ij}|^2\leq Q_{ii}Q_{jj}=0$ for every $j$. The objective depends on
$\alpha_i$ only through $-\alpha_i$, and its first update sets it to
$C$ permanently. Such coordinates may be removed before scaling. The
original gap is the remaining gap plus the nonnegative sum of their
linear gaps, while the endpoint of the first sweep has none of these
linear gaps. Thus a contraction bound for the remaining coordinates also
bounds the original first sweep and all later sweeps. If all coordinates
are removed, one sweep solves the problem.

\subsection{The common optimal gradient and quadratic growth}
\label{app:growth}

Let $B=H^{1/2}$ denote the symmetric positive semidefinite square root,
so $H=B^\top B$. Fix one optimizer $x_*$, put $w_*=Bx_*$ and
$g_*=Hx_*-a$, and define
\[
 J=\{i:g_{*,i}\ne0\},\qquad
 b_i=\begin{cases}0,&g_{*,i}>0,\\r_i,&g_{*,i}<0\end{cases}
 \quad(i\in J).
\]
The coordinates in $J$ have a nonzero optimal derivative.

We next show that $Bx_*$ and the gradient $Hx_*-a$ have the same values at
every optimizer, and use this fact to describe the solution set and the
objective gap. Related common-gradient and polyhedral-description results are
classical for convex quadratic programs \citep{mangasarian1988}.
More generally, \citet[Section~4, especially Theorem~10]{necoara2019} derive
the common-image property and Hoffman-based quadratic growth for objectives of the form
$g(Ax)+c^\top x$ over polyhedra. Their setting contains our problem with
$A=B$, $g(u)=\|u\|^2/2$, and $c=-a$; compactness of the box turns their
sublevel result into global quadratic growth. We include the short specialized
argument because it also gives the exact decomposition
\eqref{eq:decomposition} used below.

\begin{lemma}
\label{lem:geometry}
The vectors $w_*$ and $g_*$ are independent of the choice of optimizer.
Moreover,
\begin{align}
 X_*&=\{x\in\cB:Bx=w_*,\ x_i=b_i\ (i\in J)\},
 \label{eq:optset}\\
 \Delta(x)&=\tfrac12\|B(x-x_*)\|^2
             +\sum_{i\in J}|g_{*,i}|\,|x_i-b_i|.
 \label{eq:decomposition}
\end{align}
\end{lemma}
\begin{proof}
If $x_*$ and $z_*$ are optimal, their midpoint is feasible and
\[
 f\bigl((x_*+z_*)/2\bigr)
 =\tfrac12 f(x_*)+\tfrac12 f(z_*)-\tfrac18\|B(x_*-z_*)\|^2.
\]
Optimality forces $Bx_*=Bz_*$. Multiplying by $B^\top$ proves equality
of their gradients. Coordinate optimality at a box minimizer gives
$g_{*,i}\geq0$ at zero, $g_{*,i}\leq0$ at $r_i$, and $g_{*,i}=0$
in the interior. Every optimizer therefore has $x_{*,i}=b_i$ for
$i\in J$. Expand the quadratic at $x_*$:
\[
 \Delta(x)=\tfrac12\|B(x-x_*)\|^2+g_*^\top(x-x_*).
\]
For a positive derivative the corresponding linear term is
$|g_{*,i}|x_i$; for a negative derivative it is
$|g_{*,i}|(r_i-x_i)$. This gives \eqref{eq:decomposition}. All its
terms are nonnegative, so its zero set is exactly \eqref{eq:optset}.
\end{proof}

Define the largest global quadratic-growth constant by
\begin{equation}
 \nu_*:=\inf_{x\in\cB\setminus X_*}
          \frac{2\Delta(x)}{\operatorname{dist}(x,X_*)^2}.
 \label{eq:nu}
\end{equation}
It is finite: $f$ is not constant on the full-dimensional box because
$H_{ii}=1$, so a nonoptimal feasible point gives a finite upper bound.
It is also positive, as the following argument shows.

Hoffman's bound \citep{hoffman1952} states that the distance to a
nonempty linear system is bounded by a finite constant times its
constraint residual. Applied relative to a polyhedron $P$, it yields
\[
 \operatorname{dist}(x,\{z\in P:Az=c\})\leq\Theta\|Ax-c\|
 \quad(x\in P).
\]
Indeed, in the full Hoffman bound the inequalities defining $P$ have
zero residual at such $x$. Positive row scalings of $A$ preserve
finiteness. Applying this result to \eqref{eq:optset}, with appropriate
row scalings, gives a finite $\Theta>0$ satisfying
\begin{equation}
 \operatorname{dist}(x,X_*)^2\leq\Theta^2
 \left(\tfrac12\|B(x-x_*)\|^2+
       \sum_{i\in J}\frac{|g_{*,i}|}{r_i}(x_i-b_i)^2\right).
 \label{eq:hoffman}
\end{equation}
Feasibility gives $(x_i-b_i)^2/r_i\leq|x_i-b_i|$. Thus the
parenthesis is at most \eqref{eq:decomposition}, proving
\begin{equation}
 \nu_*\geq2/\Theta^2>0.
 \label{eq:positive-growth}
\end{equation}
This is an existence argument; its linear system uses unknown optimal
data. Characterizations of Hoffman constants relative to a reference
polyhedron are developed by \citet{pena2021}. Our computable screening
construction avoids requiring that constant as input.

The simpler spectral bound $\nu_*\geq\lambda_{\min}(H)$ follows from
$\Delta(x)\geq\tfrac12(x-\bar x)^\top H(x-\bar x)$.
If $H\succ0$ and the optimizer is interior, equality holds: a sufficiently
small displacement from the optimizer in a minimum-eigenvalue direction
is feasible and has zero linear term in its gap.

\section{Specializations of the contraction theorem}
\label{app:specializations}

\subsection{Scalar growth and monotonicity}

If $0<\nu\leq\nu_*$, Theorem~\ref{thm:matrix} applies to $M=\nu I$.
At $\theta=1/2$,
$(H+(\nu/2)I)^{-1}\preceq(2/\nu)I$, so
\[
 A_{1/2}(\nu I)
 =2\lambda_{\max}\!\left(T(H+(\nu/2)I)^{-1}T^\top\right)
 \leq4\|T\|^2/\nu.
\]
This proves the scalar estimate in the main text for arbitrary solution sets.

For positive definite $G$, completing the square gives
\begin{equation}
 v^\top G^{-1}v=\max_{z\in\R^p}\{2v^\top z-z^\top Gz\}.
 \label{eq:inverse-variational}
\end{equation}
Consequently $G_1\succeq G_2\succ0$ implies
$G_1^{-1}\preceq G_2^{-1}$, by comparison inside the maximum.
If $M_1\succeq M_2$ and both are admissible in Theorem~\ref{thm:matrix},
apply this fact to $H+\theta M_1$ and $H+\theta M_2$, then take the
largest quadratic form on unit vectors after multiplying by $T$ and
$T^\top$. This proves
\begin{equation}
 A_\theta(M_1)\leq A_\theta(M_2).
 \label{eq:matrix-monotonicity}
\end{equation}
In particular, if $\nu=\lambda_{\min}(M_{\rm scr})>0$, the matrix
bound is no worse than the scalar replacement $\nu I$ for every
admissible $\theta$.

\subsection{Proof of Theorem~\ref{thm:pd}}
\label{app:pd}

Suppose $H\succ0$, so there is a unique optimizer $x_*$. The argument in the
proof of Theorem~\ref{thm:matrix} remains valid at $\theta=0$ with $G=H$ and
yields
\[
 \Delta(\mathcal C(x))
 \leq\tfrac12 A_0\|\mathcal C(x)-x\|^2.
\]
Combining this inequality with \eqref{eq:decrease} proves
\eqref{eq:pd-rate}. Now assume $x_*$ is interior, so $g(x_*)=0$. Each
unprojected intermediate update is a continuous function of the starting
point and fixes $x_*$. Since there are finitely many updates, all are interior
for starting points in a sufficiently small neighborhood of $x_*$.

For such a sweep let $s=\mathcal C(x)-x$ and $e=\mathcal C(x)-x_*$.
Every updated partial derivative is zero. Equations
\eqref{eq:chronology} and \eqref{eq:decrease} therefore hold as
\begin{equation}
 He=T^\top s,\qquad
 \Delta(\mathcal C(x))=\tfrac12s^\top TH^{-1}T^\top s,\qquad
 \Delta(x)-\Delta(\mathcal C(x))=\tfrac12\|s\|^2.
 \label{eq:free-sweep}
\end{equation}
The scalar update equations also give
\[
 s_i=-[H(x-x_*)]_i-\sum_{j<i}H_{ij}s_j,
 \quad\text{hence}\quad (I+T)s=-H(x-x_*).
\]
Both $H$ and $I+T$ are invertible. Choose any nonzero top eigenvector
$s$ of $TH^{-1}T^\top$ and scale it sufficiently small that
$x=x_*-H^{-1}(I+T)s$ generates a wholly interior sweep. Then
\eqref{eq:free-sweep} gives
\[
 \frac{\Delta(\mathcal C(x))}{\Delta(x)}
 =\frac{s^\top TH^{-1}T^\top s}
        {\|s\|^2+s^\top TH^{-1}T^\top s}
 =\frac{A_0}{1+A_0}.
\]
This proves equality, including $T=0$, when every ratio is zero.

For completeness, the identity connecting this expression with
\citet[Claim~B.1]{sunye2021} follows by setting $L=I+T$ and observing
$L=H-T^\top$ and $L^\top=H-T$. Thus
\[
 L^\top H^{-1}L=(H-T)H^{-1}(H-T^\top)
 =H-T-T^\top+TH^{-1}T^\top=I+TH^{-1}T^\top.
\]
This establishes equality of the \emph{one-sweep objective factors}. The
iterate-map spectral radius and a trajectory's asymptotic objective rate are
distinct quantities.

When $H\succ0$, the infimum of $A_\theta(M)$ is unchanged when $\theta=0$
is included. Indeed,
\[
 (H+\theta M)^{-1}-H^{-1}
 =-\theta(H+\theta M)^{-1}MH^{-1}\longrightarrow0,
\]
and the largest eigenvalue is continuous on symmetric matrices. Including
the endpoint explicitly is nevertheless useful in a finite parameter
search.

\subsection{Proof of Proposition~\ref{prop:error-bound} and Wang--Lin comparison}
\label{app:error-bound}

\begin{proof}[Proof of Proposition~\ref{prop:error-bound}]
We first prove \eqref{eq:eb-rate}. Let $y=\mathcal C(x)$ and retain the
sweep-step vector $s$ from the proof of Theorem~\ref{thm:matrix}.

Projection onto an interval is nonexpansive. Coordinate optimality also
gives $y_i=\Pi_{[0,r_i]}(y_i-g_i(x^i))$ at $y=\mathcal C(x)$.
Subtracting this equality from the residual definition at $y$ yields
\[
 |R_i(y)|\leq|g_i(y)-g_i(x^i)|=|(T^\top s)_i|,
 \qquad\|R(y)\|\leq\|T\|\|s\|.
\]
Drop the nonpositive curvature term in \eqref{eq:end-gap} and use the
assumed global residual error bound:
\begin{align*}
 \Delta(y)&\leq\|T\|\|s\|\operatorname{dist}(y,X_*)\\
 &\leq\kappa\|T\|\|s\|\|R(y)\|
 \leq\kappa\|T\|^2\|s\|^2.
\end{align*}
Combining with \eqref{eq:decrease} proves the first claim of
the proposition.

For the second claim, let $\nu>0$ be any quadratic-growth constant and
fix $x\in\cB$. Put $z=\Pi_{\cB}(x-g(x))$, $R=x-z$, and
$\bar z=\Pi_{X_*}(z)$. Projection optimality gives
$g(x)^\top(z-\bar z)\leq R^\top(z-\bar z)$. Since
$g(z)=g(x)-HR$, quadratic expansion gives
\begin{align*}
 \tfrac12(\nu+\lambda_{\min}(H))\|z-\bar z\|^2
 &\leq\Delta(z)+\tfrac12(z-\bar z)^\top H(z-\bar z)\\
 &=g(z)^\top(z-\bar z)\\
 &\leq ((I-H)R)^\top(z-\bar z)\\
 &\leq\|I-H\|\|R\|\|z-\bar z\|.
\end{align*}
If $z\notin X_*$, divide by $\|z-\bar z\|$; otherwise the resulting
distance bound is immediate. The triangle inequality now yields
\[
 \operatorname{dist}(x,X_*)
 \leq\|x-z\|+\operatorname{dist}(z,X_*)
 \leq\left(1+\frac{2\|I-H\|}
 {\nu+\lambda_{\min}(H)}\right)\|R(x)\|.
\]
Thus $\kappa_\nu$ in \eqref{eq:kappa-conversion} is a valid global
residual error-bound constant, and the final claim follows by substituting
it into \eqref{eq:eb-rate}.
\end{proof}

The preceding conversion is a quantitative refinement of the general
proximal-gradient implication in \citet[Corollary~3.6]{drusvyatskiylewis2018}.
Applied with unit step to the smooth quadratic plus the indicator of $\cB$,
their result gives the error-bound constant
\[
 \kappa_{\rm DL}=\left(1+\frac{2}{\nu}\right)(1+\|H\|).
\]
Because $H\succeq0$ has unit diagonal, $\|H\|\geq1$ and
$\|I-H\|\leq\|H\|$. Hence
\[
 \kappa_\nu
 \leq1+\frac{2\|H\|}{\nu}
 <\left(1+\frac{2}{\nu}\right)(1+\|H\|)
 =\kappa_{\rm DL}.
\]
Thus \eqref{eq:kappa-conversion} strictly sharpens that general conversion
for the normalized box quadratics considered here. The comparison concerns
the constants implied by the same quadratic-growth inequality; a sharper
problem-specific residual error bound may of course be available.

To compare constants in exactly the same normalization, put
$\rho=\|H\|$ and $\tau=\|T\|$. Theorem~8 of
\citet{wanglin2014} gives the factor $\phi/(\phi+\gamma)$, with
\[
 \phi=\left(\rho+\frac{1+\beta}{\omega}\right)
       \left(1+\kappa\frac{1+\beta}{\omega}\right).
\]
Their Lemma~10 supplies $\omega=1$, $\beta=1+\rho\sqrt p$, and
$\gamma=1/2$ for exact cyclic minimization of the unit-diagonal
quadratic. Hence their factor is
\begin{equation}
 \chi_{\rm WL}=\frac{A_{\rm WL}}{1+A_{\rm WL}},\qquad
 A_{\rm WL}=2(\rho+2+\rho\sqrt p)
                 (1+\kappa(2+\rho\sqrt p)).
 \label{eq:WL}
\end{equation}
The residual here is the same unit-step Euclidean projected residual
used in their error bound. Thus the same numerical $\kappa$ is being
used, rather than differently scaled geometric information.

Let $\|H\|_F=(\sum_{i,j}H_{ij}^2)^{1/2}$ denote the Frobenius norm.
Since $\tau\leq\|H\|_F\leq\sqrt p\,\rho$,
\[
 2\kappa\tau^2\leq2\kappa p\rho^2
 <2(\rho+2+\rho\sqrt p)(1+\kappa(2+\rho\sqrt p))=A_{\rm WL}.
\]
The strict inequality follows because each factor on the right exceeds
the corresponding factor $\rho\sqrt p$ and $\kappa\rho\sqrt p$.
Since $t/(1+t)$ is strictly increasing for $t\geq0$, the bound in
\eqref{eq:eb-rate} is strictly smaller than \eqref{eq:WL}.

\subsection{Why an error bound derived from screened growth is redundant}
\label{app:dominance}

Suppose $M_{\rm scr}\succ0$, and put
\[
 \nu=\lambda_{\min}(M_{\rm scr}),\quad
 h=\lambda_{\min}(H),\quad d=\|I-H\|,\quad
 \kappa_\nu=1+\frac{2d}{\nu+h}.
\]
Since $\nu_*\geq\nu$, this $\nu$ is a valid quadratic-growth constant, so
Proposition~\ref{prop:error-bound} makes $\kappa_\nu$ a valid error-bound
constant. However,
\begin{equation}
 \inf_{0<\theta<1}\tfrac12 A_\theta(M_{\rm scr})
 \leq\frac{2\nu\|T\|^2}{(\nu+h)^2}
 \leq\kappa_\nu\|T\|^2.
 \label{eq:redundancy}
\end{equation}
To show the first inequality, use \eqref{eq:matrix-monotonicity} and
$H+\theta\nu I\succeq(h+\theta\nu)I$:
\[
 \tfrac12 A_\theta(M_{\rm scr})
 \leq\frac{\|T\|^2}{2(1-\theta)(h+\theta\nu)}.
\]
Because $M_{\rm scr}\succeq H$, one has $\nu\geq h$. The denominator
is maximized at $\theta=(\nu-h)/(2\nu)\in[0,1/2]$, as seen from
\[
 (1-\theta)(h+\theta\nu)
 =\frac{(\nu+h)^2}{4\nu}
    -\nu\left(\theta-\frac{\nu-h}{2\nu}\right)^2.
\]
If this parameter is zero then $h=\nu>0$, so the endpoint is valid.
For the second inequality in \eqref{eq:redundancy}, unit diagonal and
positive semidefiniteness give $0\leq h\leq1$ and $d\geq1-h$.
Consequently
\[
 \kappa_\nu-\frac{2\nu}{(\nu+h)^2}
 \geq1+\frac{2(1-h)}{\nu+h}-\frac{2\nu}{(\nu+h)^2}
 =\frac{\nu^2+2h-h^2}{(\nu+h)^2}>0.
\]
The last inequality in \eqref{eq:redundancy} is strict if $T\ne0$.
For a finite parameter search, include the displayed minimizing
parameter or retain its explicit upper bound to guarantee dominance over the
converted error-bound estimate.

\subsection{Elementary estimates for the triangular norm}

With $\rho=\|H\|$, unit diagonal gives $\operatorname{tr}H=p$.
If $\lambda_1,\ldots,\lambda_p$ are the eigenvalues of $H$, then
\begin{equation}
 \|T\|^2\leq\|T\|_F^2
 =\frac{\|H\|_F^2-p}{2}
 =\frac{\sum_i\lambda_i^2-p}{2}
 \leq\frac{p(\rho-1)}2\leq\frac{p(p-1)}2.
 \label{eq:triangular-estimates}
\end{equation}
Here $\lambda_i^2\leq\rho\lambda_i$, $\sum_i\lambda_i=p$, and
$\rho\leq p$. Sharper cyclic-CD analyses use logarithmic
triangular-truncation estimates \citep{li2018,sunye2021}; the elementary
envelopes above serve only as direct dimension-dependent estimates.
Directly retaining $T$ in \eqref{eq:A} also preserves order information
lost by \eqref{eq:triangular-estimates}. If $T=0$, symmetry and unit
diagonal imply $H=I$, and one sweep solves the separable problem.

\section{Screening refinements and objective lower bounds}
\label{app:screening}

\subsection{Certified faces and gradient ranges}

Let $S$ be coordinates already certified to equal endpoints
$b_i\in\{0,r_i\}$ at every optimizer, and let
$F=\{1,\ldots,p\}\setminus S$. The face
\begin{equation}
 P(S,b)=\{x\in\cB:x_i=b_i\ (i\in S)\}
 \label{eq:face}
\end{equation}
contains $X_*$. Initially $S$ is empty. Since each gradient component
is affine, its exact range on this face is $[\ell_i^{\rm face},h_i^{\rm face}]$,
where
\begin{align}
 \ell_i^{\rm face}&=-a_i+\sum_{j\in S}H_{ij}b_j
                       +\sum_{\substack{j\in F\\H_{ij}<0}}H_{ij}r_j,
 \label{eq:face-lower}\\
 h_i^{\rm face}&=-a_i+\sum_{j\in S}H_{ij}b_j
                       +\sum_{\substack{j\in F\\H_{ij}>0}}H_{ij}r_j.
 \label{eq:face-upper}
\end{align}
Each free summand $H_{ij}x_j$ is minimized and maximized at the
appropriate endpoint. These intervals therefore contain $g_{*,i}$.
If $\ell_i^{\rm face}>0$, certify $b_i=0$ with margin
$m_i=\ell_i^{\rm face}$. If $h_i^{\rm face}<0$, certify $b_i=r_i$
with margin $m_i=-h_i^{\rm face}$. Restrict the face and repeat;
there are at most $p$ additions to $S$.

\subsection{An objective lower bound with an auxiliary vector}

For any $v\in\R^p$, define
\begin{equation}
 d_S(v)=-\tfrac12v^\top Hv+
       \sum_{i\in S}b_i(Hv-a)_i+
       \sum_{i\in F}r_i\min\{(Hv-a)_i,0\}.
 \label{eq:dual-bound}
\end{equation}
For every $u\in P(S,b)$, completing the square and minimizing the
remaining linear function over the face give
\begin{align}
 f(u)&=\tfrac12(u-v)^\top H(u-v)-\tfrac12v^\top Hv+(Hv-a)^\top u
 \notag\\
 &\geq d_S(v)+\tfrac12(u-v)^\top H(u-v).
 \label{eq:dual-identity}
\end{align}
In particular,
\begin{equation}
 d_S(v)\leq f_*,\qquad
 \tfrac12(x_*-v)^\top H(x_*-v)\leq f_*-d_S(v).
 \label{eq:dual-radius}
\end{equation}
The construction accepts any $v\in\R^p$ and uses only products with $H$.
It is a box-quadratic realization of the gap-region principle
underlying safe screening \citep{ndiaye2017}.

For a feasible point $z$, taking $v=z$ gives
\begin{equation}
 f(z)-d_S(z)=g(z)^\top z-\sum_{i\in S}b_i g_i(z)
                          -\sum_{i\in F}r_i\min\{g_i(z),0\}.
 \label{eq:face-gap}
\end{equation}
If $z\in P(S,b)$, the fixed terms cancel and this becomes the
Frank--Wolfe gap on the face,
$\sum_{i\in F}[z_i(g_i(z))_++(r_i-z_i)(-g_i(z))_+]$.
If $z$ has not reached the certified endpoints, the full expression
\eqref{eq:face-gap} must be used. Restricting the face can only improve
$d_S(v)$, even in this case.

\subsection{Safe intervals and their combination}

Suppose a certified scalar $\ell\leq f_*$ is available. For a feasible
$z$, put $\overline\Delta(z)=f(z)-\ell$. Assume previous screening has
supplied margins $0<m_i\leq|g_{*,i}|$ on $S$, with the correct endpoints.
The budget
\begin{equation}
 \beta(z)=\overline\Delta(z)-\sum_{i\in S}m_i|z_i-b_i|
 \label{eq:beta}
\end{equation}
is nonnegative by \eqref{eq:decomposition}, and
$(z-x_*)^\top H(z-x_*)\leq2\beta(z)$. The same
Cauchy--Schwarz argument as in Section~\ref{sec:fw} gives
\begin{equation}
 g_{*,i}\in[g_i(z)-\sqrt{2\beta(z)},\ g_i(z)+\sqrt{2\beta(z)}].
 \label{eq:primal-interval}
\end{equation}
Known boundary penalties thus reduce the uncertainty radius.

An auxiliary vector provides a differently centered interval. For a
feasible incumbent $x$ and any $v$, let $E(x,v)=f(x)-d_S(v)\geq0$.
Equation~\eqref{eq:dual-radius} yields
\begin{equation}
 g_{*,i}\in[(Hv-a)_i-\sqrt{2E(x,v)},\ (Hv-a)_i+\sqrt{2E(x,v)}].
 \label{eq:aux-interval}
\end{equation}
Only $d_S(v)$ supports the radius in \eqref{eq:dual-radius} for this center;
a larger unrelated lower bound lacks the required centered inequality.

Intersect \eqref{eq:face-lower}--\eqref{eq:face-upper},
\eqref{eq:primal-interval}, \eqref{eq:aux-interval}, the Frank--Wolfe
interval from Section~\ref{sec:fw}, and any previously retained intervals.
If the intersection is $[L_i,U_i]$, then $L_i>0$ certifies endpoint
zero and margin $L_i$, while $U_i<0$ certifies endpoint $r_i$ and
margin $-U_i$. Otherwise retain margin zero unless an older certificate
already supplies a positive one. Correct certificates cannot disagree
on the sign. Using previously certified margins to form \eqref{eq:beta}
before generating new ones avoids circular reasoning.

One may retain a finite archive $\mathcal Z$ of feasible reference
points and use the maximum lower endpoint and minimum upper endpoint
over that archive. If a point is discarded, retain its certified
interval. On each pass, restrict the face after any new endpoint
certification, improve $\ell$, recompute the intervals, and retain the
best margins. Validity holds after any finite number of passes.

\subsection{Improving the auxiliary lower bound}

The function $d_S$ is concave and piecewise quadratic. Given an old
auxiliary vector $v_0$ and feasible incumbent $x$, a concrete improvement is
\begin{equation}
 v^+=v_0+t_*(x-v_0),\qquad
 t_*\in\argmax_{0\leq t\leq1}d_S(v_0+t(x-v_0)).
 \label{eq:segment}
\end{equation}
The segment includes both previous and current choices. Retain
$\ell\leftarrow\max\{\ell,d_S(v^+)\}$; initially
$\ell=d_{\varnothing}(0)=\sum_i r_i\min\{-a_i,0\}$.
For SVM-scaled data this is $-\sum_i r_i a_i$.

To solve \eqref{eq:segment}, put $d=x-v_0$ and $h=Hd$. The only
breakpoints are zeros of $(Hv_0-a)_i+t h_i$ for $i\in F$.
Between consecutive breakpoints let $I_-$ be the free coordinates with
negative gradient. The derivative of the objective along the segment is
\[
 -v_0^\top Hd+\sum_{i\in S}b_i h_i+
                    \sum_{i\in I_-}r_i h_i-t\,d^\top Hd.
\]
Evaluate the endpoints and its zero, if that zero lies in the interval
and $d^\top Hd>0$. If the curvature is zero, endpoints suffice.
After sorting breakpoints, crossing one changes only the corresponding
summands. With endpoint matrix products cached, the total work is
$O(p\log p)$. Tied breakpoints can be processed together.

Further tightening is possible by intersecting $P(S,b)$ with all the
regions
\[
 \tfrac12(u-z)^\top H(u-z)\leq\beta(z)\ (z\in\mathcal Z),\qquad
 \tfrac12(u-v)^\top H(u-v)\leq E(x,v).
\]
This nonempty compact convex intersection contains $X_*$. Minimizing
$g_i(u)$ and $-g_i(u)$ on it gives the strongest gradient interval
implied by this particular intersection. These optional convex refinements
can tighten the explicit interval procedure.
Certified outer bounds on their optima are needed if numerical solves
are used for safe screening.

\subsection{Monotonicity, limiting margins, and sublevel sets}

The interval retention rule makes each $m_i$ nondecreasing. Since
$m_i\leq|g_{*,i}|$,
\begin{equation}
 H\preceq M_{\rm scr}\preceq
 M_*:=H+2\diag(|g_{*,1}|/r_1,\ldots,|g_{*,p}|/r_p).
 \label{eq:limit-matrix}
\end{equation}
Whenever the inverse formula is available, \eqref{eq:matrix-monotonicity}
shows that each $A_\theta(M_{\rm scr})$, and its infimum over
$\theta$, is nonincreasing.

If the archive includes feasible $z^k$ with
$\overline\Delta(z^k)\to0$, then
$|g_i(z^k)-g_{*,i}|\leq\sqrt{2\overline\Delta(z^k)}$.
Both endpoints of the unrefined interval converge to $g_{*,i}$.
A nonzero optimal derivative is therefore certified after finitely
many rounds, with its retained margin converging to its absolute value.
A zero derivative cannot pass a strict safe sign test. It follows that
$M_{\rm scr}\to M_*$. The coefficient $|g_{*,i}|/r_i$ is the largest
constant $w$ for which $|g_{*,i}|t\geq wt^2$ throughout
$0\leq t\leq r_i$, by taking $t=r_i$. Thus $M_*$ characterizes the limit
of this coordinatewise construction; other growth constructions may yield
larger matrices.

If $M_{\rm scr}$ is singular, any zero direction has zero quadratic
form under $H$ as well, since all terms in
$v^\top M_{\rm scr}v=v^\top Hv+2\sum_i(m_i/r_i)v_i^2$ are
nonnegative. Thus $H+\theta M_{\rm scr}$ is singular for every
$\theta$. A separately known positive scalar growth constant can still
be used, as proved in Appendix~\ref{app:growth}.

For a trajectory with a known feasible incumbent $x$, suppose
$B_0=f(x)-\ell>0$. On the sublevel set $f(u)\leq f(x)$,
\eqref{eq:decomposition} gives
$m_i|u_i-b_i|\leq\Delta(u)\leq B_0$ for each positive margin.
Hence $|u_i-b_i|\leq\min\{r_i,B_0/m_i\}$. Define
\begin{equation}
 M_{\rm lev}=H+2\diag(w_1,\ldots,w_p),\qquad
 w_i=\begin{cases}
 m_i/\min\{r_i,B_0/m_i\},&m_i>0,\\0,&m_i=0.
 \end{cases}
 \label{eq:level-matrix}
\end{equation}
Repeating \eqref{eq:screen-growth} gives
$\Delta(u)\geq\tfrac12(u-x_*)^\top M_{\rm lev}(u-x_*)$ on this
sublevel set. Since $M_{\rm lev}\succeq M_{\rm scr}$, it can improve
the bound for all subsequent sweeps. Theorem~\ref{thm:matrix} uses
growth only at the sweep endpoint, which stays in this set by descent.
The resulting guarantee is a trajectory-sublevel bound rather than a global
bound for the supremum \eqref{eq:chi}. If $B_0=0$, the incumbent is already
optimal.

\subsection{A singular example requiring gap information}

Take
\[
 H=\begin{bmatrix}1&1\\1&1\end{bmatrix},\qquad
 a=(1,2)^\top,\qquad (r_1,r_2)=(2,1).
\]
These satisfy $a_i r_i=2$. Full-box gradient ranges are $[-1,2]$ and
$[-2,1]$, so face screening alone certifies no endpoint. Let
$x=(1/4,1)^\top$. Then $f(x)=-47/32$ and $g(x)=(1/4,-3/4)^\top$.
Writing $\sigma=v_1+v_2$, the full-box lower bound is
\[
 d_{\varnothing}(v)=-\tfrac12\sigma^2+
                          2\min\{\sigma-1,0\}+\min\{\sigma-2,0\}.
\]
On the segment from zero to $x$, $0\leq\sigma\leq5/4$. Its derivative
is $3-\sigma>0$ below one and $1-\sigma\leq0$ above one, so the
maximum is at $v=(1/5,4/5)^\top$ with lower bound $-3/2$.
Now $E(x,v)=1/32$, $Hv-a=(0,-1)^\top$, and the auxiliary radius is
$1/4$. This certifies $g_{*,2}\leq-3/4$ and gives
\[
 M_{\rm scr}=\begin{bmatrix}1&1\\1&5/2\end{bmatrix}\succ0.
\]
The feasible point $(0,1)^\top$ has objective $-3/2$, confirming that
the lower bound is exact. Centering \eqref{eq:primal-interval} at
$x$ with the same lower bound gives the weaker margin $m_2=1/2$;
retaining their intersection keeps the stronger one.

\section{Exact two-coordinate calculations}
\label{app:singular}

\subsection{Proof of Theorem~\ref{thm:singular}}

For the data \eqref{eq:family}, completing the square gives
\[
 f(x)=\tfrac12(x_1+x_2-1)^2-\tfrac12-\delta x_2.
\]
Since $x_2\leq r$, the minimum is at least $-1/2-\delta r$.
Equality requires both $x_2=r$ and $x_1+x_2=1$, giving the unique
feasible optimizer $(1-r,r)^\top$ and gradient $(0,-\delta)^\top$.
This also proves the gap formula in \eqref{eq:family-map}.

For any starting point $x$, the first coordinate update is $y_1=1-x_2$,
which lies in $[0,1]$ because $0\leq x_2\leq r<1$. The second is
$y_2=\min\{1+\delta-y_1,r\}=\min\{x_2+\delta,r\}$.
Write
\[
 t=r-x_2,\qquad h=\min\{\delta,t\},\qquad b=1-x_2-x_1.
\]
The displacement is $(b,h)^\top$, and direct substitution gives
\begin{equation}
 \Delta(x)=\tfrac12b^2+\delta t,\qquad
 \Delta(y)=\tfrac12h^2+\delta(t-h),\qquad
 \|y-x\|^2=b^2+h^2.
 \label{eq:family-cycle}
\end{equation}
If $t=0$, the endpoint is optimal. For fixed $t>0$, the endpoint gap
is independent of $b$, so the ratio $\Delta(y)/\Delta(x)$ is largest
at $b=0$. This choice is feasible for every $t\in[0,r]$, because
$x_1=1-x_2\in[1-r,1]$. Therefore
\begin{equation}
 \max_{x_1}\frac{\Delta(\mathcal C(x))}{\Delta(x)}=
 \begin{cases}
 t/(2\delta),&0<t\leq\delta,\\
 1-\delta/(2t),&\delta\leq t\leq r.
 \end{cases}
 \label{eq:family-profile}
\end{equation}
The two branches agree at $t=\delta$ and are nondecreasing. Since
$r>\delta$, the maximum occurs at $t=r$, $b=0$, namely
$x=(1,0)^\top$. Its value is $1-\delta/(2r)$.

For the trajectory, let $x^{(k)}=\mathcal C^k((1,0)^\top)$ and
$N=\lceil r/\delta\rceil$. Parenthesized superscripts here count
whole sweeps. Iterating \eqref{eq:family-map} gives
\begin{equation}
 x_2^{(k)}=\min\{k\delta,r\},\qquad
 x_1^{(k)}=1-\min\{(k-1)\delta,r\},\qquad k\geq1.
 \label{eq:trajectory}
\end{equation}
For $1\leq k<N$ the coordinate sum is $1+\delta$, and hence
\begin{equation}
 \Delta(x^{(k)})=\delta r-(k-\tfrac12)\delta^2.
 \label{eq:trajectory-gap}
\end{equation}
At $k=N$, put $h_N=r-(N-1)\delta\in(0,\delta]$. Then
$x_2^{(N)}=r$ but $x_1^{(N)}=1-r+h_N$, so
$\Delta(x^{(N)})=h_N^2/2>0$. The next first-coordinate update reaches
the optimizer. No earlier iterate is optimal: before sweep $N$ the
second coordinate is below $r$, and at sweep $N$ the first is too large.
Thus exactly $N+1$ sweeps are required. For
$1\leq k\leq\lfloor r/(2\delta)\rfloor<N$,
\eqref{eq:trajectory-gap} gives
$\Delta(x^{(k)})>\delta r/2=\Delta(x^{(0)})/2$.

For the spectral bound, abbreviate $\nu=\delta/r$ within this
calculation. The exact margin gives
\[
 M_{\rm scr}=\begin{bmatrix}1&1\\1&1+2\nu\end{bmatrix},\qquad
 H+\theta M_{\rm scr}=
 \begin{bmatrix}1+\theta&1+\theta\\
                 1+\theta&1+\theta+2\theta\nu\end{bmatrix}.
\]
The latter determinant is $2(1+\theta)\theta\nu>0$ and the first
diagonal entry of its inverse is
$1/(1+\theta)+1/(2\theta\nu)$. Since
$T=\bigl[\begin{smallmatrix}0&0\\1&0\end{smallmatrix}\bigr]$,
this entry is the only nonzero eigenvalue of
$T(H+\theta M_{\rm scr})^{-1}T^\top$. Therefore
\begin{equation}
 A_\theta(M_{\rm scr})=
 \frac1{1-\theta^2}+\frac1{2\nu\theta(1-\theta)}.
 \label{eq:family-A}
\end{equation}
At $\theta=1/2$ it is $4/3+2/\nu$. Substituting into
\eqref{eq:rate} gives $1-3\nu/(6+7\nu)$, exactly
\eqref{eq:family-bound}. Finally,
\[
 \frac{3\nu/(6+7\nu)}{1-\chi_*}
 =\frac{3\nu/(6+7\nu)}{\nu/2}
 =\frac6{6+7\nu}\longrightarrow1.
\]
This proves all statements of Theorem~\ref{thm:singular}.

\subsection{Exact growth and gap per squared displacement}

The same family permits two additional sharp calculations. First,
$\nu_*=\delta/r$ in \eqref{eq:nu}. To prove this, put
$\nu=\delta/r$, $d_1=x_1-(1-r)$, and $t=r-x_2$. Feasibility gives
$t\geq0$ and $d_1\leq r$. Expanding shows
\begin{equation}
 2\Delta(x)-\nu\|x-x_*\|^2
 =(1-\nu)(d_1-t)^2+2\nu t(r-d_1)\geq0.
 \label{eq:family-growth}
\end{equation}
Here $0<\nu<1$. At $x=(1,0)^\top$, $d_1=t=r$, so equality
holds at a nonoptimal point and the constant is exact.

Second, define the gap per squared sweep displacement by
\begin{equation}
 \Gamma:=\sup_{\substack{x\in\cB\\\mathcal C(x)\ne x}}
      \frac{\Delta(\mathcal C(x))}{\|\mathcal C(x)-x\|^2}.
 \label{eq:Gamma}
\end{equation}
A fixed point of $\mathcal C$ is optimal: each coordinate changes
only once, so zero total displacement means that every coordinate is
optimal at the same point. Thus this definition excludes exactly the
zero-displacement optimal sweeps. Equation~\eqref{eq:decrease} implies
$\chi_*\leq2\Gamma/(1+2\Gamma)$ whenever $\Gamma<\infty$.
Theorem~\ref{thm:matrix} gives
$\Gamma\leq\inf_{0<\theta<1}A_\theta(M)/2$.

Using \eqref{eq:family-cycle}, for fixed $t>0$ the ratio defining
$\Gamma$ is again maximal at $b=0$. It equals $1/2$ when
$t\leq\delta$, and $t/\delta-1/2$ when $t\geq\delta$.
Consequently
\begin{equation}
 \Gamma=\frac r\delta-\frac12=\frac1{\nu_*}-\frac12,\qquad
 \chi_*=\frac{2\Gamma}{1+2\Gamma}.
 \label{eq:family-Gamma}
\end{equation}
Equation~\eqref{eq:family-A} also shows
\begin{equation}
 \Gamma\leq\tfrac12\inf_{0<\theta<1}A_\theta(M_{\rm scr})
 \leq\tfrac12 A_{1/2}(M_{\rm scr})
 =\frac1{\nu_*}+\frac23=\Gamma+\frac76.
 \label{eq:Gamma-sharpness}
\end{equation}
Thus the bound has the correct leading constant for the diverging
gap/displacement ratio as well as for the vanishing decrease fraction.

\subsection{Why a pseudoinverse does not repair the Hessian-only formula}

The Moore--Penrose pseudoinverse of a positive semidefinite matrix
reciprocates its positive eigenvalues and is zero on its null space.
For \eqref{eq:family}, $H^\dagger=H/4$, so
$\lambda_{\max}(TH^\dagger T^\top)=1/4$.
Replacing $H^{-1}$ by $H^\dagger$ in \eqref{eq:pd} would predict
$\chi_*\leq1/5$, contradicting \eqref{eq:family-exact}.

The failed proof step is square completion over the whole vector space.
At the maximizing start the displacement is $s=(0,\delta)^\top$,
so $T^\top s=(\delta,0)^\top$. Along $e=t(1,-1)^\top$,
\[
 s^\top Te-\tfrac12e^\top He=\delta t,
\]
which is unbounded above as a function of $t\in\R$. A pseudoinverse
cannot represent this maximum because the linear term has a component
in the Hessian null space. The additional boundary term in
$M_{\rm scr}$ controls precisely this direction.

\subsection{An interior positive-definite comparison}

For $0\leq\eta<1$, consider
\[
 H_\eta=\begin{bmatrix}1&\eta\\\eta&1\end{bmatrix},\qquad
 a=\bm1,\qquad \cB=[0,C]^2,\qquad C>\frac1{1+\eta}.
\]
The optimizer is $x_*=\bm1/(1+\eta)$, and
$\nu_*=1-\eta$ by the interior minimum-eigenvalue argument in
Appendix~\ref{app:growth}. The exact residual error-bound constant is
$\kappa_*=1/(1-\eta)$. Indeed, for
$z=\Pi_{\cB}(x-g(x))$, nonexpansiveness and feasibility of $x_*$ give
\[
 \|z-x_*\|\leq\|(I-H_\eta)(x-x_*)\|=\eta\|x-x_*\|.
\]
The triangle inequality then gives
$(1-\eta)\|x-x_*\|\leq\|x-z\|=\|R(x)\|$.
For sufficiently small displacements $x-x_*=t(1,-1)^\top$,
projection is inactive and $R(x)=(1-\eta)(x-x_*)$, proving sharpness.

The error-bound specialization \eqref{eq:eb-rate}, even with this exact
constant, gives $2\eta^2/(1-\eta+2\eta^2)$. In contrast,
\[
 TH_\eta^{-1}T^\top=
 \begin{bmatrix}0&0\\0&\eta^2/(1-\eta^2)\end{bmatrix},
 \qquad\chi_*=\eta^2
\]
by Appendix~\ref{app:pd}. At $\eta=0.9$, these factors are approximately
$0.942$ and $0.81$. Thus separating interaction and geometric constants
can lose information even when each scalar constant is sharp.

\section{Witness search and numerical verification}
\label{app:witness}

\subsection{A feasible reference and stable objective differences}

The lower bound in \eqref{eq:bracket} uses a feasible reference $z$,
with $f_*\leq f(z)$. For starts satisfying $f(\mathcal C(x))>f(z)$,
write
\begin{equation}
 q_z(x)=\frac{f(\mathcal C(x))-f(z)}{f(x)-f(z)}.
 \label{eq:witness-ratio}
\end{equation}
For $A=f(\mathcal C(x))$ and $B=f(x)\geq A$, the ratio
$(A-t)/(B-t)$ is nonincreasing for $t<A$.
Since $f_*\leq f(z)<A$, this proves $0<q_z(x)\leq\chi_*$.
For a fixed start with positive endpoint gap,
improving $z$ toward optimality increases this lower estimate toward
the true ratio, once its numerator is positive.

To reduce cancellation, evaluate both objective differences with
\begin{equation}
 f(u)-f(z)=g(z)^\top(u-z)+\tfrac12(u-z)^\top H(u-z).
 \label{eq:stable-gap}
\end{equation}
The linear term must be retained even if the reference is close to an
optimizer, and also at an exact boundary optimizer.

\subsection{Forward and reverse evaluation of a sweep}

Maintain the current iterate $u$, gradient $g=Hu-a$, and objective
$F=f(u)$. At coordinate $i$, compute
\[
 v_i=u_i-g_i,\qquad u_i^+=\Pi_{[0,r_i]}(v_i),\qquad
 s_i=u_i^+-u_i.
\]
Using the pre-update values, replace
$F\leftarrow F+g_i s_i+s_i^2/2$, then update $u_i$ and
$g\leftarrow g+s_iH_{\cdot i}$, where $H_{\cdot i}$ is column $i$.
This is exact in real arithmetic and costs $O(p^2)$ for a dense sweep.

Suppose no pre-clipping argument $v_i$ equals a box endpoint. Define
$d_i=1$ if $0<v_i<r_i$ and $d_i=0$ otherwise. These strict
inequalities persist near the starting point, so the sweep is affine
on that neighborhood. Because $H_{ii}=1$,
$v_i=a_i-\sum_{j\ne i}H_{ij}u_j$ does not depend on the old $u_i$.
The derivative matrix $D_i$ of update $i$ has identity rows except for
row $i$, where
\[
 (D_i)_{ii}=0,\qquad (D_i)_{ij}=-d_iH_{ij}\quad(j\ne i).
\]
The derivative of the full sweep is $J=D_p\cdots D_1$.
To compute $J^\top w$, process $i=p,p-1,\ldots,1$. At each step
save $c=w_i$, set $w_j\leftarrow w_j-d_iH_{ij}c$ for $j\ne i$,
and set $w_i\leftarrow0$. This implements multiplication by
$D_i^\top$ in the required reverse order.

Starting with $w=g(\mathcal C(x))$, the final vector is
$b=J^\top g(\mathcal C(x))$. Holding $z$ fixed, the quotient rule
then gives
\begin{equation}
 \nabla q_z(x)=
 \frac{(f(x)-f(z))b-(f(\mathcal C(x))-f(z))g(x)}{(f(x)-f(z))^2}.
 \label{eq:witness-gradient}
\end{equation}
The reverse pass costs $O(p^2)$ and requires only vectors and clipping
indicators beyond the stored Hessian. At clipping thresholds,
\eqref{eq:witness-gradient} need not be a gradient. A search can
perturb to a neighboring differentiable region, but every proposed step
must be evaluated using the actual clipped sweep. Derivatives from one
region must not be extended across a threshold without reevaluation.

\subsection{Candidate starts and local improvement}

Useful starts include uniform points in the box and points concentrated
near its endpoints, together with their next few cyclic iterates.
An additional heuristic uses coordinates of the reference $z$ that
are close to an endpoint: rank them by increasing $|g_i(z)|$, move a
selected coordinate to its opposite endpoint, and approximately minimize
over the others while keeping that coordinate fixed. A small derivative
suggests a weak boundary penalty. Use this rule only for initialization;
activity certification still comes from safe intervals. For the singular
family, using its exact
optimizer as reference and moving the second coordinate to zero recovers
the worst start $(1,0)^\top$ after minimizing the first coordinate.

If $H\succ0$ and $z$ is interior, the equality proof in
Appendix~\ref{app:pd} suggests starting in directions
$-H^{-1}(I+T)s$, where $s$ is a top eigenvector of $TH^{-1}T^\top$.
Sample feasible distances in both signs. Similar directions from a
principal submatrix on approximately interior coordinates may be useful.
Treat every such point as a start and score it with the full clipped map.
Box-constrained local optimization of $q_z$ can then improve the best
candidates using \eqref{eq:witness-gradient} where it is valid.
Retain the best evaluated feasible point, including the initial point
if local improvement fails.

Searching may exclude very small denominators for numerical stability.
Such exclusions can reduce tightness; every retained, correctly evaluated
witness still gives a valid lower bound. A zero search result is merely the
trivial lower bound because finite local search may miss the global maximum
of \eqref{eq:witness-ratio}.

\subsection{Verifying the two endpoints of the contraction interval}

For a rigorous upper bound, objective lower bounds must be enclosed
downward, gap bounds upward, and gradient intervals outward. A strictly
positive lower endpoint or strictly negative upper endpoint is needed
to certify a margin. A negative numerical value for the theoretically
nonnegative budgets $\beta(z)$ or $E(x,v)$ signals enclosure error and
requires correction. Verify positive definiteness and enclose the largest
eigenvalue in \eqref{eq:A} upward. Use a verified upper enclosure rather
than a raw eigenvalue estimate from an unconverged iteration. Any verified
admissible value of $\theta$ supplies a valid bound.

For the lower bound, first verify feasibility of the retained start
$x$ and reference $z$. Obtain enclosures
\[
 F_U\geq f(z),\qquad A_L\leq f(\mathcal C(x)),\qquad B_U\geq f(x).
\]
The endpoint enclosure must propagate error through the \emph{whole sweep}
before the final objective is evaluated. One option
is to propagate intervals through the updates and then evaluate the
quadratic on the resulting box. Scalar clipping is monotone, so an
input interval $[v_L,v_U]$ maps into
$[\Pi_{[0,r_i]}(v_L),\Pi_{[0,r_i]}(v_U)]$ even across thresholds.

If $A_L>F_U$, then
\begin{equation}
 \frac{A_L-F_U}{B_U-F_U}
 \leq\frac{f(\mathcal C(x))-F_U}{f(x)-F_U}
 \leq\frac{f(\mathcal C(x))-f_*}{f(x)-f_*}\leq\chi_*.
 \label{eq:verified-witness}
\end{equation}
The first inequality reduces a positive numerator and increases a
positive denominator. The second uses the ratio's monotonicity in
the subtracted scalar and $F_U\geq f(z)\geq f_*$.
Report a downward enclosure of the left side, including rounding in
both subtractions and the division. If the strict test fails, refine
the enclosures or use another witness; zero remains valid.

\subsection{Computational costs and available checks}

For dense $H$, a forward sweep and a reverse pass each cost $O(p^2)$,
with $O(p)$ working storage beyond the Hessian. If $K$ starts are
evaluated and local refinement uses $N$ trial evaluations in total,
the evaluation cost is $O((K+N)p^2)$. Obtaining the reference,
constructing relaxed starts, and spectral factorizations are additional
costs. Dense spectral calculations can cost $O(p^3)$ with $O(p^2)$
storage. Sparse forward and reverse passes cost $O(p+\operatorname{nnz}(H))$,
where $\operatorname{nnz}(H)$ is the number of nonzero Hessian entries.

Face-gradient ranges require $O(p^2)$ initialization. Adding a newly
certified coordinate changes one term in every interval, so all face
updates together cost $O(p^2)$. With $q$ archived points and cached
gradients, each gap-interval pass costs $O(pq)$; the segment optimization
adds $O(p\log p)$ with cached products. These repeated passes,
positive-definiteness checks, and optional quadratic-region solves are
separate costs; cache and reuse their results while the inputs stay unchanged.

Deterministic numerical checks cover lower bounds, safe intervals,
matrix growth, and cycle estimates, including singular and nonunique
examples. They also check
the sharp singular family and its trajectories with exact rational
arithmetic. Full certification of the implementation requires interval
enclosures in addition to these floating-point checks.
\section{Experimental protocol and additional results}
\label{app:experiments}

\subsection{Data and computation}

The experiments were run with Python 3.12.5 on macOS 15.6.1.
Table~\ref{tab:datasets} lists the six scaled binary classification
datasets from the LIBSVM collection.
The labels are mapped to $\{-1,1\}$; no train--test split is made,
because the quantity of interest is the contraction of the training
objective. The linear kernel is $K(x,x')=x^\top x'$ and the RBF kernel is
$K(x,x')=\exp(-\gamma\|x-x'\|^2)$. The five RBF values of $\gamma$
are multiples of $\gamma_0$ defined in Section~\ref{sec:experiments}.
For each kernel matrix $Q_{ij}=y_i y_j K(x_i,x_j)$, we use the diagonal
change of variables from Appendix~\ref{app:scaling} and perform exact
clipped coordinate updates in the file order. Thus $p$ is the number of
examples, and one sweep contains $p$ updates.

\begin{table}[t]
\centering
\begin{tabular}{lrrr}
\toprule
Dataset & Examples $p$ & Features & $\gamma_0$ \\
\midrule
Sonar & 208 & 60 & $0.0526$ \\
Heart & 270 & 13 & $0.0857$ \\
Australian & 690 & 14 & $0.1007$ \\
Diabetes & 768 & 8 & $0.7202$ \\
Fourclass & 862 & 2 & $1.0636$ \\
German.numer & 1000 & 24 & $0.0481$ \\
\bottomrule
\end{tabular}
\caption{Dataset dimensions and median-distance RBF scale. The value
$\gamma_0$ is computed from all nonzero off-diagonal squared distances
within each dataset.}
\label{tab:datasets}
\end{table}

For each configuration, cyclic CD first supplies a feasible reference,
stopping at a projected-gradient residual of $10^{-8}$ or at 20,000
sweeps. When needed, the implementation refines this point with a
box-constrained optimizer and further sweeps. The screened matrix uses
the larger of two valid margins for each coordinate: iterative
face-gradient intervals from Appendix~\ref{app:screening} and the
Frank--Wolfe-gap margin \eqref{eq:fw-margin} evaluated at the reference.
The scalar $\theta$ in \eqref{eq:A} is optimized numerically. For the
Wang--Lin comparison, the implementation uses
$\nu=0.99\lambda_{\min}(M_{\rm scr})$ and the valid, conservative
error-bound constant $\kappa=1+2\|I-H\|/\nu$ in \eqref{eq:WL}.
When $H$ is numerically positive definite, the implementation also
compares the spectral result with \eqref{eq:pd-rate} and retains the
stronger factor.

The observed lower estimate of $\chi_*$ maximizes a one-sweep
objective-gap ratio over starts sampled uniformly in the box, near its
endpoints, and around a reference solution. The search also includes
iterates of sampled starts, weak-boundary perturbations, and spectral
directions when applicable. It refines promising starts with local
box-constrained optimization and reevaluates the actual clipped sweep.
The largest ratio is reported as $\chi_{\rm obs}$. Its exact-arithmetic
lower-bound interpretation requires only a feasible reference, as in
\eqref{eq:bracket}. The finite local search need not find the global
maximum, and floating-point evaluation is not an interval-verified
lower bound. Likewise, the reported
spectral factors use floating-point eigenvalues and screening margins.

A dash or crossed heatmap cell means that a required numerical test did
not succeed. In particular, the code declines to report
$\chi_{\rm scr}$ when the computed smallest eigenvalue of
$M_{\rm scr}$ is below its positive-definiteness tolerance; the
Wang--Lin value derived from this matrix is then also unavailable.
The Gauss--Seidel comparison is omitted when $H$ fails its numerical
positive-definiteness test. This includes every linear-kernel case in
Table~\ref{tab:linear-main} and Table~\ref{tab:linear-app}.

\subsection{Additional RBF results}

Figure~\ref{fig:rbf-app} reports the remaining RBF grids. The
screened factor is available for all 25 diabetes and all 25
german.numer configurations. For fourclass it is unavailable in four
configurations, all with $C\geq10$ and $\gamma\leq\gamma_0$.
The visible distance between the two plotted values persists on
these datasets. Across the 71 appendix RBF configurations with both
screening and Wang--Lin values, the screened decrease fraction is
larger in every case; its median ratio to Wang--Lin is
$3.8\times10^8$. The corresponding median ratio to Gauss--Seidel
over 35 available comparisons is $50.3$.

\begin{figure}[t]
\centering
\includegraphics[width=\linewidth]{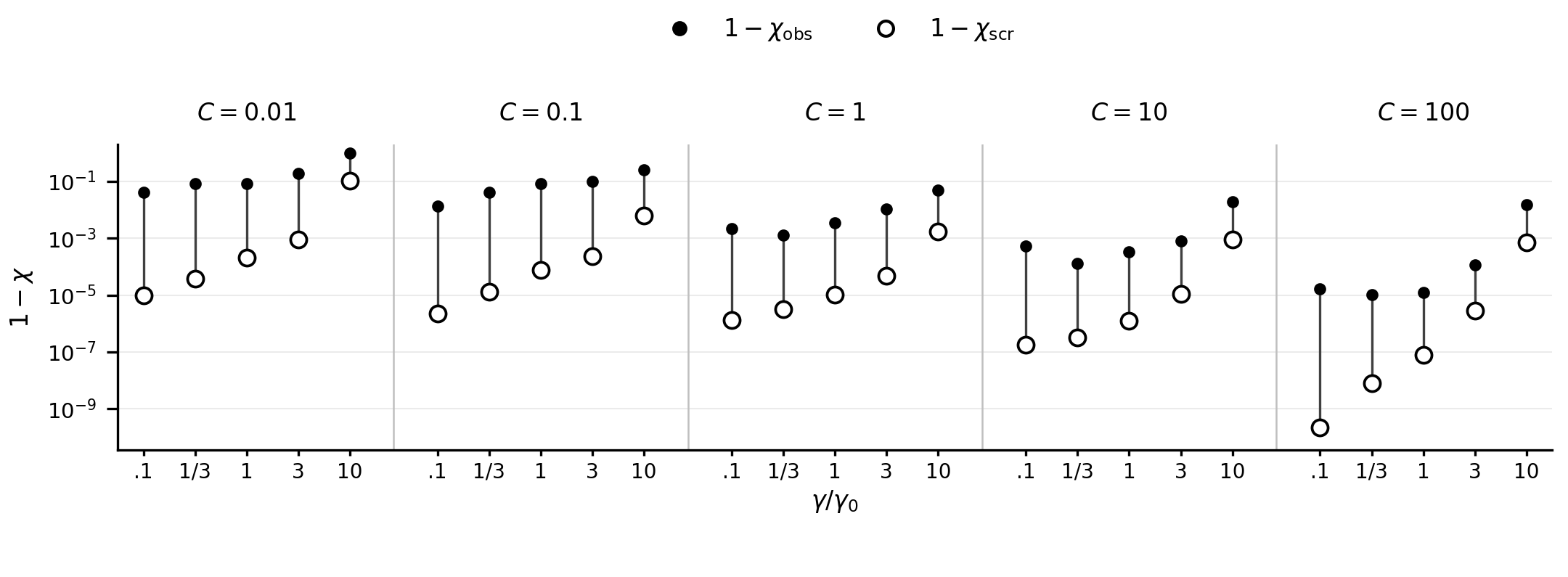}
\includegraphics[width=\linewidth]{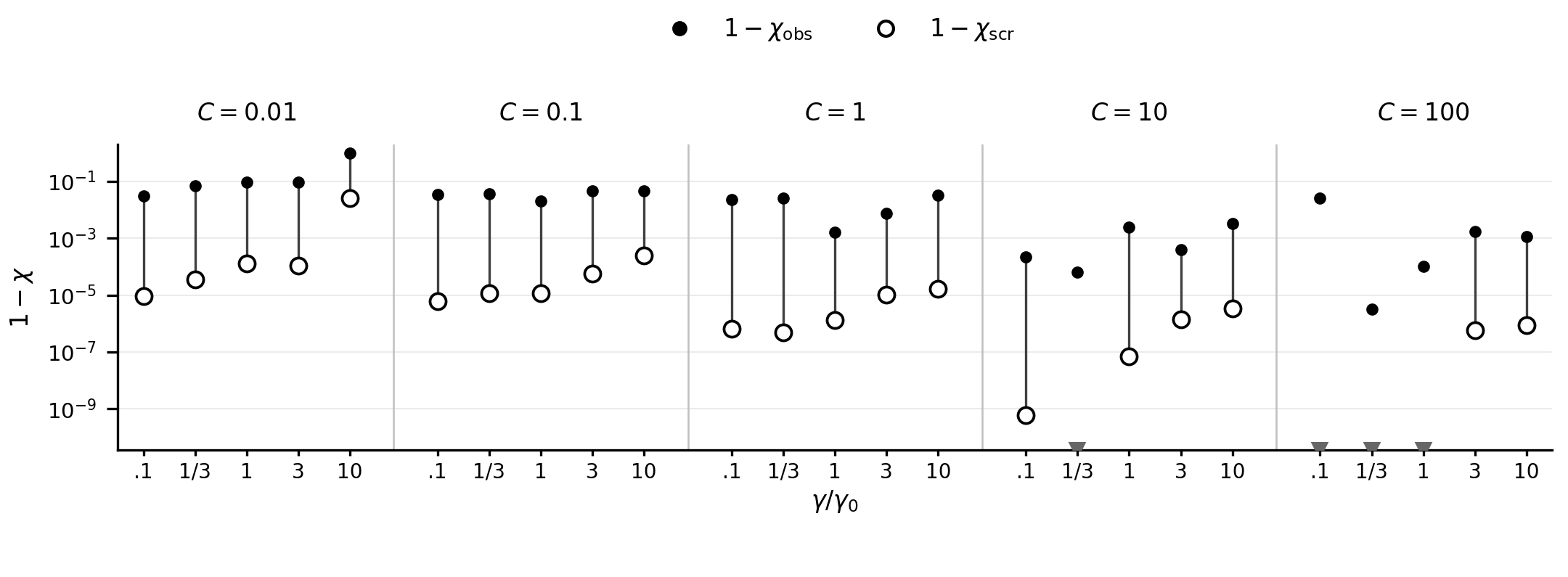}
\includegraphics[width=\linewidth]{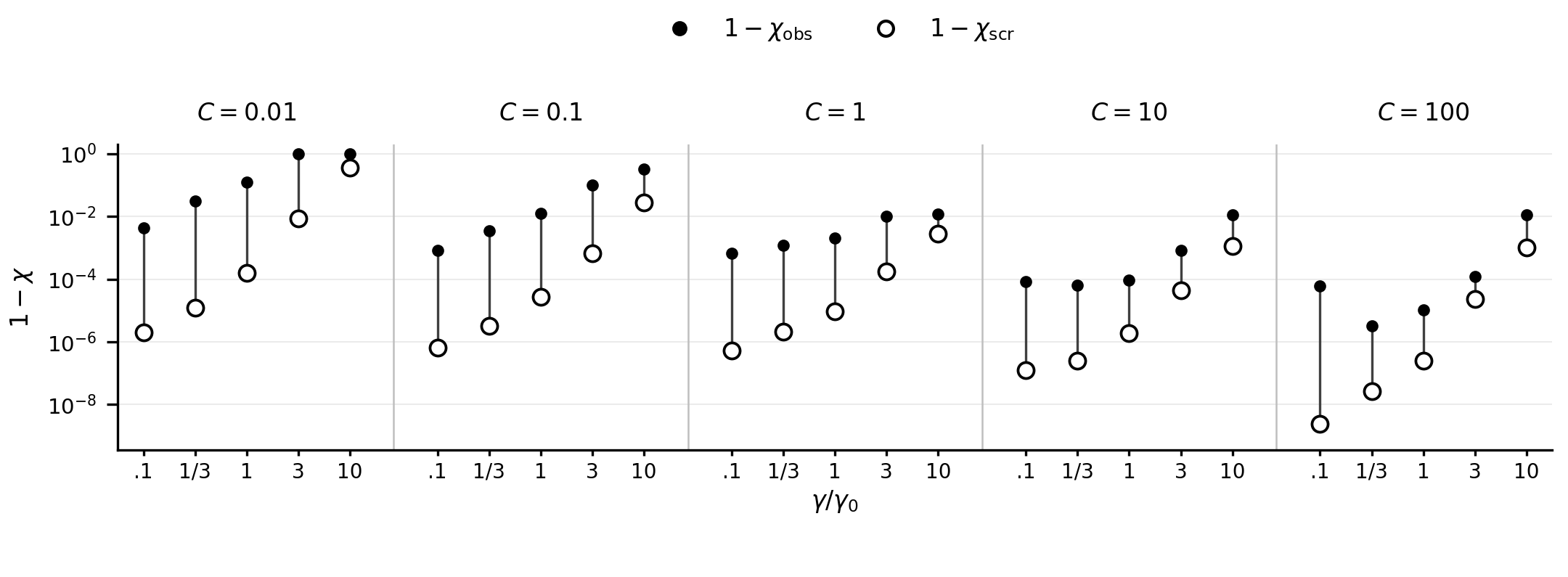}
\caption{RBF one-sweep decrease fractions for diabetes (top), fourclass
(middle), and german.numer (bottom). Filled circles give the observed
search value $1-\chi_{\rm obs}$; open circles give the screened
guarantee $1-\chi_{\rm scr}$. Horizontal ticks give
$\gamma/\gamma_0$ within each labeled $C$ group, and the vertical
scale is logarithmic.}
\label{fig:rbf-app}
\end{figure}

\begin{figure}[t]
\centering
\includegraphics[width=\linewidth]{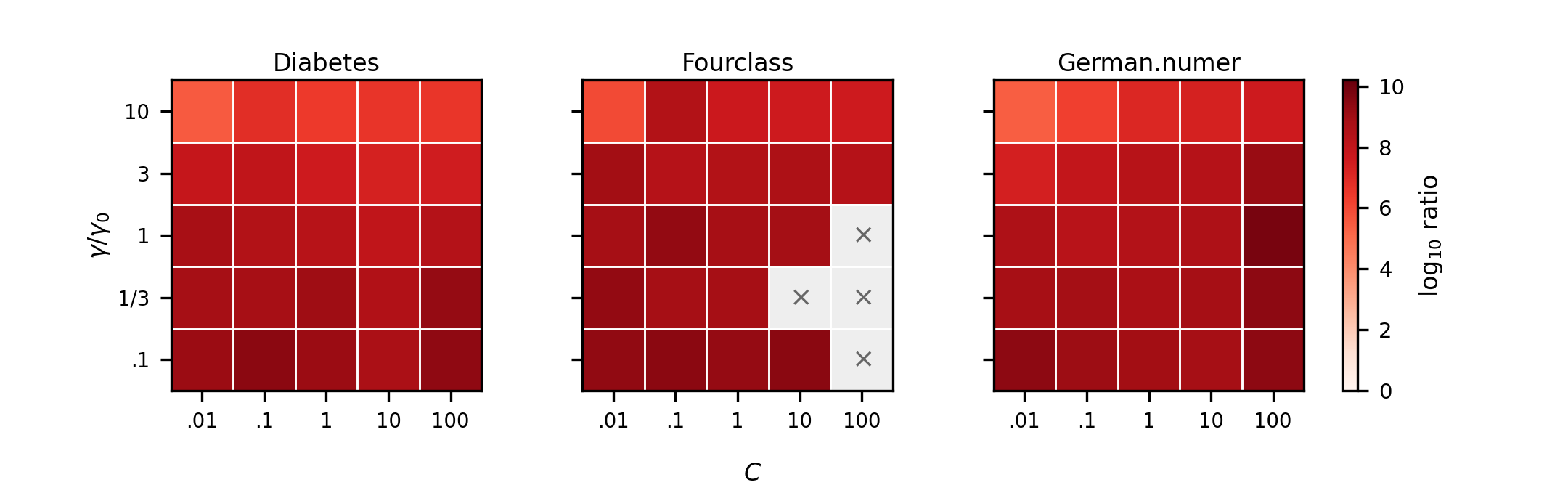}
\caption{RBF certificate improvement on Wang--Lin for diabetes,
fourclass, and german.numer. Each cell shows
$\log_{10}[(1-\chi_{\rm scr})/(1-\chi_{\rm WL})]$;
crosses mark unavailable comparisons.}
\label{fig:heatmaps-app}
\end{figure}

\subsection{Additional linear-kernel results}

Table~\ref{tab:linear-app} gives the remaining linear-SVM values.
The screened matrix passes its numerical test for every diabetes and
fourclass configuration and for german.numer except at $C=100$.
As in the main-text table, the computed screening guarantee is much
stronger than the Wang--Lin value while remaining below the observed
decrease estimate.

\begin{table}[t]
\centering
\small
\begin{tabular}{llrrr}
\toprule
Dataset & $C$ & $1-\chi_{\rm obs}$ & $1-\chi_{\rm scr}$ & $1-\chi_{\rm WL}$ \\
\midrule
Diabetes & $0.01$ & $2.77\times10^{-2}$ & $3.59\times10^{-5}$ & $6.31\times10^{-14}$ \\
 & $0.1$ & $3.09\times10^{-2}$ & $3.50\times10^{-6}$ & $2.87\times10^{-15}$ \\
 & $1$ & $8.91\times10^{-3}$ & $3.49\times10^{-7}$ & $2.09\times10^{-16}$ \\
 & $10$ & $9.37\times10^{-3}$ & $6.26\times10^{-8}$ & $2.44\times10^{-17}$ \\
 & $100$ & $2.16\times10^{-2}$ & $7.00\times10^{-9}$ & $2.34\times10^{-18}$ \\
\midrule
Fourclass & $0.01$ & $3.03\times10^{-1}$ & $1.68\times10^{-4}$ & $1.84\times10^{-13}$ \\
 & $0.1$ & $4.38\times10^{-2}$ & $1.24\times10^{-5}$ & $4.94\times10^{-15}$ \\
 & $1$ & $3.98\times10^{-2}$ & $2.66\times10^{-6}$ & $2.97\times10^{-15}$ \\
 & $10$ & $5.66\times10^{-3}$ & $3.32\times10^{-7}$ & $4.07\times10^{-16}$ \\
 & $100$ & $8.92\times10^{-3}$ & $3.31\times10^{-8}$ & $4.06\times10^{-17}$ \\
\midrule
German.numer & $0.01$ & $1.02\times10^{-2}$ & $1.97\times10^{-6}$ & $3.53\times10^{-15}$ \\
 & $0.1$ & $1.23\times10^{-2}$ & $3.08\times10^{-7}$ & $4.20\times10^{-16}$ \\
 & $1$ & $7.77\times10^{-4}$ & $6.80\times10^{-10}$ & $1.73\times10^{-19}$ \\
 & $10$ & $1.92\times10^{-4}$ & $3.76\times10^{-10}$ & $2.68\times10^{-19}$ \\
 & $100$ & $1.26\times10^{-2}$ & -- & -- \\
\bottomrule
\end{tabular}
\caption{Linear-SVM decrease fractions for the three appendix datasets.
Notation and the meaning of a dash are as in Table~\ref{tab:linear-main}.}
\label{tab:linear-app}
\end{table}

\subsection{Coordinate permutations}

For the permutation study in Figure~\ref{fig:permutations}, we hold
the sonar data and each of the three $(C,\gamma)$ or linear-kernel
settings fixed. We draw 100 permutations with seeds 42 through 141
and permute $H$, $a$, the box widths, and the reference solution
together before recomputing both estimates. The first RBF case uses
the large box $C=100$: a numerical box-constrained fit has 204 free
coordinates and four at zero using a $10^{-6}$ endpoint tolerance.
The second RBF case uses $C=1$ and has
substantially more active box constraints. Both Hessians are
positive definite, and both optima lie on the boundary. The linear
case has a singular Hessian. The gray lines in the figure pair
results from the same permutation, while its bars summarize the
middle 50\% of the 100 values. The plotted spread reflects order
sensitivity, not uncertainty from repeated data sampling.

\end{document}